\documentclass{amsart}
\usepackage{amsthm}
\newtheorem{theorem}{Theorem}[section]
\newtheorem{lemma}[theorem]{Lemma}
\newtheorem{proposition}[theorem]{Proposition}
\theoremstyle{definition}
\newtheorem{definition}[theorem]{Definition}

\theoremstyle{remark}
\newtheorem{remark}[theorem]{Remark}

\usepackage{amssymb}
\usepackage{empheq}
\usepackage{stmaryrd}
\usepackage{enumerate}
\usepackage[shortlabels]{enumitem}
\usepackage{calc}
\usepackage{url}

\usepackage{xcolor}

\newcommand{\norm}[1]{\left\lVert#1\right\rVert}
\usepackage{mathtools}

\DeclarePairedDelimiter{\floor}{\lfloor}{\rfloor}

\usepackage[left=2.0cm,%
                right=2.0cm,%
                top=2.5cm,%
                bottom=3.5cm,%
                headheight=12pt,%
                a4paper]{geometry}%

\begin{document}

\title{Logical metatheorems for accretive and (generalized) monotone set-valued operators}

\author[Nicholas Pischke]{Nicholas Pischke}
\date{\today}
\maketitle
\vspace*{-5mm}
\begin{center}
{\scriptsize Department of Mathematics, Technische Universit\"at Darmstadt,\\
Schlossgartenstra\ss{}e 7, 64289 Darmstadt, Germany, \ \\ 
E-mail: pischkenicholas@gmail.com}
\end{center}

\maketitle
\begin{abstract}
Accretive and monotone operator theory are central branches of nonlinear functional analysis and constitute the abstract study of set-valued mappings between function spaces. This paper deals with the computational properties of certain large classes of operators, namely accretive and (generalized) monotone set-valued ones. In particular, we develop (and extend) for this field the theoretical framework of proof mining, a program in mathematical logic that seeks to extract computational information from prima facie `non-computational' proofs from the mainstream literature. To this end, we establish logical metatheorems that guarantee and quantify the computational content of theorems pertaining to accretive and (generalized) monotone set-valued operators. On one hand, our results unify a number of recent case studies, while they also provide characterizations of central analytical notions in terms of proof theoretic ones on the other, which provides a crucial perspective on needed quantitative assumptions in future applications of proof mining to these branches.
\end{abstract}
\noindent
{\bf Keywords:} Proof mining; Metatheorems; Accretive operators; (Generalized) monotone operators.\\ 
{\bf MSC2010 Classification:} 03F10, 03F35, 47H05, 47H06 

\section{Introduction}

\subsection{Motivation and summary}

For both practical and conceptual reasons, it is an interesting question what the computational content of a given mathematical theorem is. \emph{Proof mining} is a program in mathematical logic that seeks to extract computational information from prima facie `non-computational' proofs from the mainstream literature and in essence, this paper extends the state-of-the-art of the underlying logical approach to proof mining to be applicable for proofs from \emph{accretive} and \emph{monotone operator theory}, central branches of nonlinear functional analysis which constitute the abstract study of set-valued mappings between function spaces. In particular, we establish so-called logical metatheorems that guarantee and quantify the computational content of theorems pertaining to accretive and (generalized) monotone set-valued operators. These extensions of the previous logical results are expected to lead to many new case studies for proof mining of results from accretive and monotone operator theory.

\smallskip

In more detail, particular emphasis has recently been placed on problems involving such set-valued operators like in \cite{Koh2019b} in the context of Bauschke's solution \cite{Bau2003} to the zero displacement conjecture, in \cite{KKA2015} for abstract Cauchy problems, in \cite{KP2020} for iteration schemes using set-valued operators or in particular like in the case of the proximal point algorithm (see \cite{Mar1970,Roc1976}) and its adaptions and extensions as treated in \cite{DP2020,DP2021,Koh2020,Koh2021b,Koh2021,KLN2018,LAS2018,LP2021,LS2018,Pin2021,Sip2021}. These \emph{case studies} in proof mining provide quantitative results for some of the most prominent results from this area. However, in some ways they are also \emph{ad hoc} in that they are not captured by known metatheorems. It is thus a pressing issue in proof mining to establish a new metatheorem for treating the aforementioned operators. 
The main result of this paper is the establishment of such a metatheorem.

\smallskip

For the rest of this section, we provide a brief history of proof mining and how it leads up to this paper (Section \ref{briho}), as well as a more detailed overview of the contents of this paper (Section \ref{klink}).

\subsection{A brief history of proof mining}\label{briho}

We provide a brief overview of the key (historical) aspects of the proof mining program with a focus on the  \emph{logical metatheorems}, the logical `substrate' of this discipline.

\smallskip

First of all, proof mining as a subfield of mathematical logic emerged in the later 1990's and early 2000's through the works of U. Kohlenbach and his collaborators (going back conceptually to Kreisel's \emph{unwinding of proofs} from the 1950's, see \cite{Kre1951,Kre1952}) as an applied discipline which uses well-known proof interpretations like negative translations, Kreisel's \emph{modified realizability} and G\"odel's \emph{functional (Dialectica) interpretation} on actual mathematical theorems to extract explicit quantitative information like (uniform) witnesses or bounds. 

For instance, in the case of convergence statements, the logical results guarantee the existence of highly uniform \emph{rates of metastability} in the sense of T. Tao \cite{Tao2008b,Tao2008a} in very general situations and thus provide a logical perspective on Tao's `finitary analysis'. The proofs analyzed are, as common in ordinary mathematical practice, prima facie noneffective which makes this a nontrivial task. The development of proof mining is detailed comprehensively up to 2008 in the monograph \cite{Koh2008} (see also \cite{KO2003} for a survey of the early stages of proof mining) and recent progress, with a focus on nonlinear analysis and optimization, is surveyed in \cite{Koh2019}.

\smallskip

In that vein, proof mining is crucially supported by the previously mentioned \emph{general logical metatheorems}\footnote{Examples of such metatheorems may be found in \cite{GeK2008,GuK2016,Koh2005,KN2017,Leu2006,Leu2014,Sip2019}, as well as \cite{Koh2008}, for the metatheorems obtained via (modifications of) G\"odel's Dialectica interpretation, and \cite{FLP2019} for subsequent metatheorems obtained via the bounded functional interpretation \cite{FO2005} due to F. Ferreira and P. Oliva.} which guarantee the existence of quantitative information for large\footnote{For the metatheorems to apply, the theorems at hand need to confine to a certain logical form and possess proofs satisfying some restrictions on the principles involved.  Nonetheless, these restrictions still allow applications to a large numbers of actual proofs from the literature involving a wide range of non-computational `ideal principles', provided of course that these theorems and their corresponding proofs can be formalized (at least in theory) in the corresponding language.} classes of theorems and proofs from the literature, including nonlinear optimization and analysis. We stress that the proofs analyzed in proof mining may use classical logic and other `non-constructive' principles. Besides merely guaranteeing the existence of quantitative information, the metatheorems allow for an a priori estimation of their complexity (which can be as elementary as polynomials) and they provide an algorithmic approach towards actually extracting the quantitative information.

\smallskip

The first metatheorems in proof mining relied on pure systems of arithmetic in all finite types and consequently only covered applications involving Polish metric spaces as those can be represented in the underlying language. More modern systems include symbols for abstract metric and normed spaces (originating in \cite{GeK2008,Koh2005}) into the language of the underlying logics. This new approach opens the door for treating spaces which are not separable and thus not representable in the (bare) language of finite type arithmetic. In this way, the following spaces have been successfully studied in proof mining: general metric and normed spaces, so-called $W$-hyperbolic spaces, $\mathrm{CAT}(0)$-spaces, uniformly convex spaces and Hilbert spaces, among many others. Whether a class of spaces can be treated via metatheorems ultimately depends on the complexity and uniformity of the defining axioms. Besides the (much) greater scope, the new approach via abstract spaces also yields extremely \emph{uniform} bounds, being independent from most parameters appearing in the theorem being analyzed. 

\subsection{Extending the scope of proof mining to set-valued operators}\label{klink}

Towards establishing our new metatheorems, we first introduce new formal systems that extend the previously used ones for normed and inner product spaces. This is done via carefully selected constants and corresponding axioms that allow for the formalization of proofs involving  accretive, monotone and $\rho$-comonotone operators and their resolvents. We show how key parts of the theory of these operators can be formally carried out in these systems. In particular, we characterize the key property of an operator being maximal by equivalent notions involving formal extensionality of the operator. This new point of view provides crucial insight into the (uniform) quantitative assumptions that have to be placed on an operator if one wants to treat proofs involving essential applications of those maximality principles. This culminates in establishing general logical metatheorems for these systems (and suitable extensions), which in particular provide a `logical' explanation of the aforementioned case studies. 

\smallskip

The application of proof mining to concrete mathematical proofs can only be successful if our logical systems have a certain modularity in the following sense: specific problems may require us to extend our `main' logical system with specific mathematical objects or notions and associated axioms, all the while guaranteeing that our metatheorems still hold. As examples of such extensions, we shall discuss certain (formalized versions of) some common notions from mathematical practice which are crucial in the previously mentioned case studies, in particular discussing quantitative forms of extensionality for a set-valued operator, as well as range conditions and treating so-called selection functions, i.e. functionals $a:X\to X$ with $ax\in Ax$ for $x\in\mathrm{dom}A$ for a given set-valued operator $A$. In that context, motivated by logical aspects of the latter, we also introduce a new notion of majorizability for set-valued operators 

\smallskip

Our main guiding principle for the design of the aforementioned logical systems, as well as for the choice of extensions considered later, is provided by the previously discussed proof mining case studies in the context of accretive, monotone and generalized monotone operators over linear spaces. The applicability of the metatheorems established later will then in particular be justified by the fact that those case studies can indeed be recognized as applications of these metatheorems  and they thus provide the first logical explanation for these results. Further discussions regarding the use of the systems introduced here to (previous) case studies will also be given in \cite{KP2022}.

\smallskip

In conclusion, we expect that our metatheorems will be applicable for a wide range of new case studies involving accretive, monotone and generalized monotone operators. In particular, we strongly believe that the general methodological approach for treating set-valued operators and their resolvents and the new corresponding notions of, e.g., majorizable operators introduced in the later sections may serve as a basis for further metatheorems in the context of nonlinear analysis involving set-valued operators.

\section{Set-valued operators in Banach and Hilbert spaces}\label{sec:optheory}

We begin by surveying the basic notions and results for accretive and (generalized) monotone operators over normed and inner product spaces. Let $(X,\norm{\cdot})$ be a normed space, which will always be a real linear space in this paper.

\subsection{Properties of convex and nonexpansive functions}
Before moving on to set-valued operators, we first recall some fundamental notions for functions on normed and inner product spaces. For that, we follow the definitions of \cite{BR1977} (and \cite{BMW2020} regarding conically nonexpansive functions). Let $D\subseteq X$ be nonempty and let $T:D\to X$ be a function. Then $T$ is called
\begin{enumerate}
\item[(1)] \emph{nonexpansive} if
\[
\forall x,y\in D\left(\norm{Tx-Ty}\leq\norm{x-y}\right),
\]
\item[(2)] \emph{firmly nonexpansive} if
\[
\forall x,y\in D\forall r>0\left(\norm{Tx-Ty}\leq\norm{r(x-y)+(1-r)(Tx-Ty)}\right),
\]
\item[(3)] \emph{$\alpha$-averaged} if $\alpha\in (0,1)$ and
\[
T=(1-\alpha)Id+\alpha N
\]
for some nonexpansive $N:D\to X$,
\item[(4)] \emph{$\alpha$-conically nonexpansive} if $\alpha\in (0,\infty)$
\[
T=(1-\alpha)Id+\alpha N
\]
for some nonexpansive $N:D\to X$.
\end{enumerate}
In the case of $\alpha$-averaged or $\alpha$-conically nonexpansive operators, we will often use the trivially equivalent reformulation that
\[
(1-\alpha^{-1})Id+\alpha^{-1}T
\]
is nonexpansive for the respective $\alpha$.

\smallskip

There are various useful equivalent reformulations of these notions when we pass to inner product spaces which we describe in the following remark.
\begin{remark}\label{rem:altDefNE}
Let $X$ be an inner product space with inner product $\langle\cdot,\cdot\rangle$ and induced norm $\norm{\cdot}$. Then $T$ is firmly nonexpansive if, and only if
\[
\forall x,y\in D\left(\langle x-y,Tx-Ty\rangle\geq\norm{Tx-Ty}^2\right),
\]
$\alpha$-averaged if, and only if $\alpha\in (0,1)$ and
\[
\forall x,y\in D\left((1-\alpha)\norm{(Id-T)x-(Id-T)y}^2\leq\alpha\left(\norm{x-y}^2-\norm{Tx-Ty}^2\right)\right),
\]
and $\alpha$-conically nonexpansive if, and only if $\alpha\in (0,\infty)$ and
\begin{align*}
&\forall x,y\in D\Big(2\alpha\langle Tx-Ty,(Id-T)x-(Id-T)y\rangle\\
&\qquad\qquad\qquad\qquad\geq (1-2\alpha)\norm{(Id-T)x-(Id-T)y}^2\Big).
\end{align*}
It can be easily seen that $T$ is firmly nonexpansive if, and only if $T$ is $1/2$-averaged.
\end{remark}
Proofs of these facts can be found in \cite{BC2017} and, respectively, \cite{BMW2020} for the $\alpha$-conically nonexpansive case. We will present formal versions of some of these proofs in the later sections in the context of formal systems for abstract normed and inner product spaces.
\subsection{Properties of set-valued operators}
A set-valued operator on a space $X$ is simply a mapping $A:X\to 2^X$.

\smallskip

For such a set-valued operator $A$, we define $\mathrm{gra} A:=\left\{(x,u)\in X\times X\mid u\in Ax\right\}$ as well as $\mathrm{dom}A:=\left\{x\in X\mid Ax\neq\emptyset\right\}$ and $\mathrm{ran}A:=\bigcup_{x\in X}Ax$. We write $A^{-1}$ for the operator defined by $x\in A^{-1}u$ iff $u\in Ax$. We set $\lambda A$ by $(\lambda A)x:=\{\lambda u\mid u\in Ax\}$. If $B$ is another set-valued operator on $X$, we define $A+B$ via $(A+B)x:=\{u+v\mid u\in Ax\land v\in Bx\}$.

\smallskip

The main classes of set-valued operators explored here are the analytically motivated accretive, monotone and $\rho$-comonotone operators. All impose some form of `separability' on the sets $Ax$ in relation to a varying $x$.

\begin{definition}
$A$ is called \emph{accretive} (as introduced in \cite{Kat1967}) if 
\[
\forall (x,u),(y,v)\in\mathrm{gra} A,\lambda >0 \left(\norm{x-y+\lambda(u-v)}\geq\norm{x-y}\right)
\]
and $A$ is called \emph{m-accretive} if $\mathrm{ran}(Id+\gamma A)=X$ for all $\gamma>0$.
\end{definition}
Accretivity is also sometimes equivalently characterized (see, e.g., \cite{Bro1967} and also \cite{Dei1985,Tak2000}), using the normalized duality mapping $J:X\to 2^{X^*}$ defined via
\[
J(x):=\left\{j\in X^*\mid \langle x,j\rangle=\norm{x}^2=\norm{j}^2\right\}
\]
where $X^*$ is the dual space of $X$. Then $A$ is accretive iff
\[
\forall (x,u),(y,v)\in\mathrm{gra}A\exists j\in J(x-y)\left( \langle u-v,j\rangle\geq 0\right).
\]

Now, for an inner product space $(X,\langle\cdot,\cdot\rangle)$, we introduced a number of monotonicity notions. 
\begin{definition}
$A$ is called \emph{monotone} (essentially due to \cite{Min1960,Min1962}) if
\[
\forall (x,u),(y,v)\in\mathrm{gra} A\left(\langle x-y,u-v\rangle\geq 0\right)
\]
and $A$ is called \emph{maximally monotone} if it is monotone and $\mathrm{gra} A\subsetneq\mathrm{gra}B$ implies that $B$ is not monotone, i.e. the graph of $A$ is not properly contained in the graph of another monotone operator. 
\end{definition}
\begin{definition}
$A$ is called \emph{$\rho$-comonotone} for $\rho\in\mathbb{R}$ if 
\[
\forall (x,u),(y,v)\in\mathrm{gra}A\left(\langle x-y,u-v\rangle\geq\rho\norm{u-v}^2\right)
\]
and, similarly to before, $A$ is called \emph{maximally $\rho$-comonotone} if it is $\rho$-comonotone and there is no proper $\rho$-comonotone extension.
\end{definition}
We are here following the definitions and the exposition of \cite{BMW2020}. For $\rho<0$, the above notions were however already considered under the name of $\vert\rho\vert$-hypocomonotonicity in \cite{CP2004} (with its dual, $\vert\rho\vert$-hypomonotonicity, already considered in \cite{RW1998}).
\subsection{Resolvents and correspondence results}\label{sec:resolventprelim}
The main tool for studying these classes of set-valued operators is their \emph{resolvent} $J^A_\gamma$, defined as follows for $\gamma>0$:
\[
J^A_\gamma:=(Id+\gamma A)^{-1}.
\]
In particular, the following defining equivalence holds:
\[
p\in J^A_\gamma x\text{ iff }\gamma^{-1}(x-p)\in Ap.
\]
By that, one can also immediately see that that $J^A_\gamma$ (as a set-valued mapping) satisfies $\mathrm{dom}J^A_\gamma=\mathrm{ran}(Id+\gamma A)$ and $J^A_\gamma x\subseteq\mathrm{dom}A$ for all $x$.

The importance of the resolvent stems mainly from two aspects: 
\begin{enumerate}
\item[(1)] Resolvents are ubiquitous in algorithmic approaches to problems in accretive/monotone operator theory due to their asymptotic properties. 
\item[(2)] Most fundamental properties of the operator $A$ correspond to fundamental and well-studied properties of the resolvent, creating a strong form of duality (see in particular \cite{BMW2012}).
\end{enumerate}
We now survey those correspondence results in the spirit of the second item as they are central to our logical investigations later on. To this end, we begin with accretive operators on normed spaces. For further (basic) results on these accretive operators and their correspondence theory to their resolvents, see \cite{Bar1976}. We, however, want to in particular highlight the following result.
\begin{theorem}[essentially \cite{Bar1976,BR1977,Tak2000}]\label{thm:accretiveEquivalence}
Let $A$ be a set-valued operator on a normed space $X$. Then the following are equivalent: 
\begin{enumerate}
\item[(a)] $A$ is accretive,
\item[(b)] $J^A_\gamma$ is single-valued and firmly nonexpansive (on its domain) for all $\gamma>0$,
\item[(c)] $J^A_\gamma$ is single-valued and firmly nonexpansive (on its domain) for some $\gamma>0$,
\item[(d)] $J^A_\gamma$ is single-valued and nonexpansive (on its domain) for all $\gamma>0$.
\end{enumerate}
\end{theorem}

The main reference for monotone operators in Hilbert spaces is the comprehensive monograph \cite{BC2017}. Regarding those monotone operators, we rely on the following correspondence result.
\begin{theorem}[essentially \cite{Bar1976,BR1977,Min1962}]\label{thm:monotoneEquivalence}
Let $X$ be a Hilbert space and $A$ a set-valued operator.
\begin{enumerate}
\item[(1)] Items (a) - (d) of Theorem \ref{thm:accretiveEquivalence} are equivalent to
\begin{enumerate}
\item[(e)] $A$ is monotone.
\end{enumerate}
\item[(2)] $A$ is maximally monotone if and only if $J^A_\gamma$ is single-valued, firmly nonexpansive and $\mathrm{ran}(\mathrm{Id}+\gamma A)=X$ for some/any $\gamma>0$. 
\end{enumerate}
\end{theorem}
The last statement is known as Minty's theorem \cite{Min1962}, a deep result in monotone operator theory. We already see that maximality conditions are linked with the totality of the resolvent, a result which sets a characteristic theme in the correspondence theory as it extends to various other classes besides monotone operators. 

Indeed, as established in the main work on correspondence theory for $\rho$-comonotone operators \cite{BMW2020}, we have an analogous result for these generalized monotone operators.
\begin{theorem}[essentially \cite{BMW2020}]\label{thm:comonotoneEquivalence}
Let $X$ be a Hilbert space and $A$ be a set-valued operator.
\begin{enumerate}
\item[(1)] The following are equivalent where is each case $\alpha=\frac{1}{2(\rho/\gamma+1)}$:
\begin{enumerate}
\item $A$ is $\rho$-comonotone.
\item $J^A_\gamma$ is single-valued and $\alpha$-conically nonexpansive for all/some $\gamma>0$ with $\rho>-\gamma$.
\item $J^A_\gamma$ is single-valued and $\alpha$-averaged for all/some $\gamma>0$ with $\rho>-\gamma/2$.
\end{enumerate}
\item[(2)] $A$ is maximally $\rho$-comonotone if, and only if 
\begin{enumerate}
\item $J^A_\gamma$ is single-valued, $\alpha$-averaged and total for some/any $\gamma>0$ such that $\rho>-\gamma/2$, or
\item $J^A_\gamma$ is single-valued, $\alpha$-conically nonexpansive and total for some/any $\gamma>0$ such that $\rho>-\gamma$,
\end{enumerate}
where $\alpha=\frac{1}{2(\rho/\gamma+1)}$.
\end{enumerate}
\end{theorem}
This correspondence between totality of the resolvent and set-theoretic maximality does not extend to accretive operators on normed spaces as first asked in \cite{CP1969} and then answered in \cite{Cal1970,CL1971} negatively. The one direction that remains valid is the following:
\begin{lemma}[essentially \cite{CP1969}]\label{lem:accretivechara}
\begin{enumerate}
\item[(1)] If $\mathrm{ran}(\mathrm{Id}+\gamma A)=X$ for some $\gamma>0$, then $A$ has no proper accretive extension.
\item[(2)] Let $A$ be accretive. If $\mathrm{ran}(\mathrm{Id}+\gamma A)=X$ for some $\gamma>0$, then $\mathrm{ran}(\mathrm{Id}+\gamma A)=X$ for all $\gamma>0$.
\end{enumerate}
\end{lemma}
\begin{proof}
We only show the first item, the second can be shown as outlined in \cite{CP1969}: Let $x,u$ be such that
\[
\forall (y,v)\in\mathrm{gra}A,\lambda\geq 0(\norm{x-y+\lambda(u-v)}\geq\norm{x-y}).\tag{$\dagger$}
\]
We want to show $u\in Ax$. By totality of $J^A_{\gamma}$, we have that $J^A_{\gamma}(x+\gamma u)$ is well defined and thus $x+\gamma u=J^A_{\gamma}(x+\gamma u)+((x+\gamma u)-J^A_{\gamma}(x+\gamma u))$. Also by definition, we have
\[
\gamma^{-1}((x+\gamma u)-J^A_{\gamma}(x+\gamma u))\in A(J^A_{\gamma}(x+\gamma u)).
\]
and this combined with ($\dagger$) implies 
\begin{align*}
0&=\norm{x-J^A_{\gamma}(x+\gamma u)+\gamma(u-\gamma^{-1}((x+\gamma u)-J^A_{\gamma}(x+\gamma u)))}\\
&\geq\norm{x-J^A_{\gamma}(x+\gamma u)}.
\end{align*}
Thus $x=J^A_{\gamma}(x+\gamma u)$ and therefore $u=\gamma^{-1}((x+\gamma u)-J^A_{\gamma}(x+\gamma u))$. This gives $u\in Ax$ and $A$ therefore is maximally accretive.
\end{proof}
In light of Theorem \ref{thm:accretiveEquivalence}, we may conclude the following: 
\begin{enumerate}
\item[(1)] If there is a $\gamma>0$ with $J^A_\gamma$ single-valued, firmly nonexpansive and total, then $A$ is maximally accretive.
\item[(2)] If $J^A_\gamma$ is single-valued, nonexpansive and total for all $\gamma>0$, then $A$ is maximally accretive.
\end{enumerate}
In particular, m-accretivity implies maximally accretive. 

Also proofs for various directions of Theorems \ref{thm:accretiveEquivalence}, \ref{thm:monotoneEquivalence} and \ref{thm:comonotoneEquivalence} will be provided in the upcoming section in the contexts of formal systems for normed and inner product spaces with accretive, monotone or $\rho$-comonotone operators.

\section{Logical systems for operators and their resolvents}
In this section, we introduce the formal systems capturing operators and their resolvents.

\smallskip

To formally treat those types of operators in systems which allow for bound extraction results, we build on the usual formal setup for proof mining as developed\footnote{We again mention \cite{GuK2016,KN2017,Leu2006,Leu2014,Sip2019} for similar approaches to bound extraction theorems and \cite{FLP2019} for metatheorems in the context of the bounded functional interpretation.} in \cite{GeK2008,Koh2005}. In this setup, one extends logics of classical arithmetic and analysis of finite type by additional types and constants to handle abstract spaces and operations on them which are not representable in the sense of representations of Polish spaces in Baire space. 

\smallskip

To that end, the next sections provides a short overview of some of the underlying notions, with a particular focus on those that play a role in the derivation of our results.

\subsection{Systems for arithmetic of finite type and extensions}\label{sec:sysoffinittypeintro}
In this section, we introduce the logical systems needed for the extensions considered later in this paper. Concretely, we first introduce the `base system' $\mathcal{A}^{\omega}$ in Section \ref{fluiter} and discuss the formalization of the real numbers in $\mathcal{A}^{\omega}$ in Section \ref{hijis}. Extensions of $\mathcal{A}^{\omega}$ to $\mathcal{A}^\omega[X,\norm{\cdot}]$, involving abstract types $X$, are then considered in Section \ref{fliere}.

\smallskip

In Section \ref{labelz}, we then further extend $\mathcal{A}^\omega[X,\norm{\cdot}]$ to accommodate the previously discussed classes of operators and their resolvents. 
\subsubsection{The base system $\mathcal{A}^{\omega}$}\label{fluiter}
We introduce the `base system' $\mathcal{A}^{\omega}=\mathrm{WE}$-$\mathrm{PA}^\omega+\mathrm{QF}$-$\mathrm{AC}+\mathrm{DC}$ for (a fragment of) classical analysis over all finite types $T$, to be extended with extra axioms in Section \ref{fliere}. All systems considered here will be extensions of the basic system $\mathcal{A}^\omega$ as defined in \cite{GeK2008,Koh2005}. We only sketch the key features of $\mathcal{A}^\omega$ and its extensions and refer to \cite{Koh2008,Tro1973} for any further details.

\smallskip

The set of types $T$ is defined as follows:
\[
0\in T,\quad \rho,\tau\in T\Rightarrow \tau(\rho)\in T.
\]
These types are stratified by their \emph{degrees}, defined recursively via
\[
\mathrm{deg}(0):=0,\quad\mathrm{deg}(\tau(\rho)):=\max\{\mathrm{deg}(\tau),\mathrm{deg}(\rho)+1\}.
\]
We make similar conventions for dropping parentheses in types as made in \cite{Koh2008} (see page 47). We use a short notation using natural numbers for the pure types $P$, given by
\[
0\in P,\quad \rho\in P\Rightarrow 0(\rho)\in P,
\]
by recursively defining $0(n):=n+1$.

\smallskip

The language of $\mathrm{WE}$-$\mathrm{PA}^\omega$/$\mathcal{A}^\omega$ is a many-sorted language with $\land,\lor,\rightarrow$ as primitives and containing quantifiers and variables for all finite types, extended with constants $0$ for zero, $S$ for successor and particular constants $\Sigma_{\rho,\tau},\Pi_{\delta,\rho,\tau}$ for the so-called \emph{combinators} as considered already by Sch\"onfinkel \cite{Sch1924} and later used extensively by Curry and Howard (see \cite{How1980} for the latter).

Further, the language contains constants $\underline{R}_{\underline{\rho}}=(R_1)_{\underline{\rho}},\dots,(R_k)_{\underline{\rho}}$ for simultaneous primitive recursion (in the sense of G\"odel \cite{Goe1958}, see also \cite{Koh2008}) for tuples of types $\underline{\rho}$. The only relation symbol is $=_0$ for equality at type $0$ and the only prime formulas are consequently $s=_0t$ for $s,t$ terms of type $0$. New terms are formed from the constants and variables only via application: if $t$ is a term of type $\tau(\rho)$ and $s$ a term of type $\rho$, then $t(s)$ is a term of type $\tau$. Higher type equality is treated as a defined notion via
\[
s=_\rho t:=\forall y_1^{\rho_1},\dots,y_k^{\rho_k}(sy_1\dots y_k=_0ty_1\dots y_k)
\]
for terms $s,t$ of type $\rho=0\rho_k\dots\rho_1$.

\smallskip

The system $\mathrm{WE}$-$\mathrm{PA}^\omega$ extends the usual finite-type variant of Peano arithmetic (see \cite{Koh2008,Tro1973}) with only the following weak \emph{rule of quantifier-free extensionality}
\[
\mathrm{QF}\text{-}\mathrm{ER}:\quad\frac{A_0\rightarrow s=_\rho t}{A_0\rightarrow r[s/x^\rho]=_\tau r[t/x^\rho]}
\]
where $A_0$ is a quantifier-free \emph{formula}, $s,t$ are terms of type $\rho$, $r$ is a term of type $\tau$ and $r[s/x^\rho]$ denotes simultaneous substitution of $s$ for all occurrences of $x$ in $r$. Later, we in particular rely on the following remark.

\begin{remark}\label{rem:strongExtRule}
One can actually derive
\[
\Sigma_1\text{-}\mathrm{ER}:\quad\frac{\exists y^\sigma A_0(y)\rightarrow s=_\rho t}{\exists y^\sigma A_0(y)\rightarrow r[s/x^\rho]=_\tau r[t/x^\rho]}
\]
with $A_0$, $s$, $t$ and $\rho,\tau$ as before and $\sigma$ an additional finite type but where we assume that $y$ is not free in $r$, $s$, $t$. This can be easily seen by noting that 
\[
\exists y^\sigma A_0(y)\rightarrow s=_\rho t\equiv\forall y^\sigma(A_0(y)\rightarrow s=_\rho t)
\]
and the latter implies $A_0(y)\rightarrow s=_\rho t$. Now, using $\mathrm{QF}$-$\mathrm{ER}$ applied to this (where it is important that $A_0$ in the formulation may have free variables), we get $A_0(y)\rightarrow r[s/x^\rho]=_\tau r[t/x^\rho]$ and universal generalization yields
\[
\forall y^\sigma(A_0(y)\rightarrow r[s/x^\rho]=_\tau r[t/x^\rho])\equiv\exists y^\sigma A_0(y)\rightarrow r[s/x^\rho]=_\tau r[t/x^\rho],
\]
which is as required. 
\end{remark}
Of course, we may actually have tuples instead of a single $y$ in (both variants of) the rule.

\smallskip

As is well-known, one can internally define $\lambda$-abstractions via the combinators in the sense that for any term $t$ of type $\tau$ and any variable $x^\rho$ of type $\rho$, there is a term $\lambda x^\rho.t$ of type $\tau(\rho)$ such that provably
\[
(\lambda x^\rho.t)(s^\rho)=_\tau t[s/x].
\]

\smallskip

The system $\mathcal{A}^\omega$ now extends $\mathrm{WE}$-$\mathrm{PA}^\omega$ by the quantifier-free version of the axiom of choice in finite types
\[
\mathrm{QF}\text{-}\mathrm{AC}:\quad\forall\underline{x}\exists\underline{y} A_0(\underline{x},\underline{y})\rightarrow\exists\underline{Y}\forall\underline{x} A_0(\underline{x},\underline{Y}\underline{x})
\]
where $A_0$ is quantifier free and the $\underline{x},\underline{y}$ may be of arbitrary type and the axiom of dependent choice $\mathrm{DC}=\{\mathrm{DC}^{\underline{\rho}}\mid\underline{\rho}\in T\}$ where
\[
\mathrm{DC}^{\underline{\rho}}:\quad\forall x^0,\underline{y}^{\underline{\rho}}\exists\underline{z}^{\underline{\rho}} A(x,\underline{y},\underline{z})\rightarrow\exists\underline{f}^{\underline{\rho}(0)}\forall x^0 A(x,\underline{f}(x),\underline{f}(S(x)))
\]
where $A$ is now of arbitrary complexity.

\subsubsection{Real numbers and related results in $\mathcal{A}^{\omega}$}\label{hijis}

We now discuss how real numbers are represented in $\mathcal{A}^{\omega}$ and discuss some essential properties. First of all, rationals and reals are represented as usual as objects of type $0$ and $1$, respectively. In that way, we follow the definitions and conventions given in \cite{Koh2008} and only present those points crucial to the development of the new metatheorems later on (together with some basic but important facts). For the coding of rationals as pairs of natural numbers, it will be convenient to fix a pairing function which we do (following the conventions of \cite{Koh2008}) by setting
\[
j(n^0,m^0):=\begin{cases}\min u\leq_0(n+m)^2+3n+m[2u=_0(n+m)^2+3n+m]&\text{if existent},\\0^0&\text{otherwise}.\end{cases}
\]
Using those codes, the operations $+_\mathbb{Q},\cdot_\mathbb{Q},(\cdot)^{-1}_\mathbb{Q}$ are primitive recursively definable and there exist quantifier-free formulas $=_\mathbb{Q}$, $<_\mathbb{Q}$ defining the respective relations.

\smallskip

Secondly, on the level of the representation of reals by fast converging Cauchy sequences with a fixed modulus $2^{-n}$ (see \cite{Koh2008}), one can then similarly define formulas $=_\mathbb{R}$/$<_\mathbb{R}$ on type $1$ which define the corresponding relations of the real numbers represented by the inputs. These relations, however, are not decidable anymore but are $\Pi^0_1$/$\Sigma^0_1$-formulas, respectively.

One also easily defines closed terms $+_\mathbb{R},\cdot_\mathbb{R},\vert\cdot\vert_\mathbb{R}$ representing the usual operations of real arithmetic on these type $1$ objects. Any natural $n$ or rational $q$ can be easily seen as a real defined just via the constant $n$- or $q$-sequence and we write $n$ or $q$, respectively, for that type $1$ representation as well.

However, the definition of the reciprocal $(\cdot)^{-1}$ in the reals is of a more delicate matter and as this features prominently in the theory later developed, we want to give a little more detail in that case. In fact, there is no closed term of type $1(1)$ in $\mathrm{WE}$-$\mathrm{PA}^\omega$ which represents $\gamma^{-1}$ correctly for all $\gamma\neq 0$. Following \cite{Koh1996}, we handle this by using a binary term $(\cdot)^{-1}_\cdot$ of type $1(1)(0)$ such that $(\gamma)^{-1}_l$ correctly represents $\gamma^{-1}$ for all $\vert\gamma\vert>2^{-l}$. An expression like $\gamma^{-1}$ is then dealt with by assuming an additional parameter $l$ of type $0$ and using $(\gamma)^{-1}_l$ together with the additional implicative assumption $\vert\gamma\vert_\mathbb{R}>_\mathbb{R}2^{-l}$. In practice, this can be mostly ignored and we thus mainly use $\gamma^{-1}$ freely without the additional parameter. However, we discuss some important practical implications of these problems with the reciprocal in Remark \ref{rem:division} later on.

Lastly, note that extensionality of the operations $+_\mathbb{R}$, $\cdot_\mathbb{R}$ and $\vert\cdot\vert_\mathbb{R}$ w.r.t. $=_\mathbb{R}$ can be proved in (weak fragments of) $\mathrm{WE}$-$\mathrm{PA}^\omega$ and in the case of $(\cdot)^{-1}$, extensionality can be shown for all $l$ and $f,g$ such that $\vert f\vert_\mathbb{R},\vert g\vert_\mathbb{R}>_\mathbb{R}2^{-l}$ (see \cite{Koh1996}). 

These type 1 representations can be similarly carried out for other Polish and, in particular, compact spaces (see \cite{Koh2008}).

In the context of the bound extraction theorems established later, we associate a canonical type 1 representation $(r)_\circ$ with a non-negative real $r\in [0,\infty)$ as introduced in \cite{Koh2005} (see also \cite{GeK2008,Koh2008})\footnote{Such an association will in general be non-effective but will behave nice enough with majorization which serves all intends and purposes.}.
\begin{definition}[\cite{Koh2005}]
For $r\in [0,\infty)$, define $(r)_\circ\in\mathbb{N}^\mathbb{N}$ via
\[
(r)_\circ(n):=j(2k_0,2^{n+1}-1),
\]
where
\[
k_0:=\max k\left[\frac{k}{2^{n+1}}\leq r\right].
\]
\end{definition}
We also cite some properties of $(\cdot)_\circ$ which will be of use later.
\begin{lemma}[\cite{Koh2005}]
\begin{enumerate}
\item[(1)] $(r)_\circ$ is a representation of $r\in [0,\infty)$ in the sense of the above.
\item[(2)] For $r,s\in [0,\infty)$, if $r\leq s$, then $(r)_\circ\leq_\mathbb{R}(s)_\circ$ and also $(r)_\circ\leq_1 (s)_\circ$.
\item[(3)] If $r\in [0,\infty)$, then $(r)_\circ$ is nondecreasing (as a type 1 function).
\end{enumerate} 
\end{lemma}
Further, we write $r_\alpha$ for the real represented by some type 1 functional $\alpha$. In the following, we omit the subscripts of the arithmetical operations for $\mathbb{R}$ and also for $X$ in the case of $\cdot_X$ (or even $\cdot_X$ altogether) to avoid notational overload. Also, we will omit types of variables whenever convenient and omit types in proofs almost always to make everything more readable.

\subsubsection{The base systems and abstract types}\label{fliere}

In this section, we extend $\mathcal{A}^{\omega}$ to $\mathcal{A}^\omega[X,\norm{\cdot}]$ using a new abstract type $X$ (following \cite{GeK2008,Koh2005}). The latter allows us to deal with abstract spaces that cannot necessarily be represented in $\mathcal{A}^\omega$/$\mathrm{WE}$-$\mathrm{PA}^\omega$. Define the extended set of types $T^X$ as follows:
\[
0,X\in T^X,\quad \rho,\tau\in T^X\Rightarrow \tau(\rho)\in T^X.
\]
The theory $\mathcal{A}^\omega$ can then be formulated over the resulting extended language by extending the constants (if appropriate) to take arguments and produce values in those new types and by trivially extending the axiom schemes and rules to allow formulas from the new language (see \cite{Koh2008} for details on all of this). 

Our new type can be used with additional constants and axioms to represent a wide range of spaces and operations on them, resulting in respective theories extending $\mathcal{A}^\omega$ (formulated over $T^X$) and a detailed discussion of various examples of such extensions can be found in \cite{GeK2008,Koh2005,Koh2008}.

The main extension used here will be the theory {$\mathcal{A}^\omega[X,\norm{\cdot}]$\label{th:BaseANorm}} for real normed vector spaces, obtained by extending $\mathcal{A}^\omega$ (formulated over $T^X$) by new constants $0_X,1_X$ of type $X$, $+_X$ of type $X(X)(X)$, $-_X$ of type $X(X)$, $\cdot_X$ of type $X(X)(1)$ and $\norm{\cdot}_X$ of type $1(X)$ together with the relevant defining axioms stating that $X$ with the operations is a real normed vector space with $1_X$ such that $\norm{1_X}_X=_\mathbb{R}1$ and $-_X x$ being the additive inverse of $x$ (see \cite{GeK2008,Koh2005,Koh2008}). Further, extensionality of all those operations is provable in $\mathcal{A}^\omega[X,\norm{\cdot}]$. It should be noted that $=_0$ is still the only primitive relation and in particular, identity on $X$ is treated as a defined predicate via
\[
x^X=_Xy^X:=\norm{x-_Xy}_X=_\mathbb{R}0
\]
which is, by the previous discussion on the representation of the reals, a $\Pi^0_1$-formula and not decidable.

Derived from $\mathcal{A}^\omega[X,\norm{\cdot}]$ is the theory {$\mathcal{A}^\omega[X,\langle\cdot,\cdot\rangle]$\label{th:BaseAInnProd}} for real inner product spaces, extending the former by the parallelogram law
\[
\forall x^X,y^X\left(\norm{x+_Xy}^2_X+\norm{x-_Xy}^2_X=_\mathbb{R}2\left(\norm{x}^2_X+\norm{y}^2_X\right)\right).
\]
As is well-known, any inner product space satisfies this law and conversely, any normed space satisfying it actually admits an inner product which can then be defined via the norm with
\[
\langle x^X,y^X\rangle_X:=_1\frac{1}{4}\left(\norm{x+_Xy}^2_X-\norm{x-_Xy}^2_X\right).
\]
Also here, extensionality of the defined operation is provable in the system.

Following \cite{GeK2008,Koh2005} (see also \cite{Koh2008}), we introduce some special notation for denoting specific classes of types from $T^X$. We call a type $\rho$ \emph{of degree $n$} if $\rho\in T$ and it has degree $\leq n$ in the usual sense. Further we call $\rho$ \emph{small} if it is of the form $\rho=\rho_0(0)\dots(0)$ (including $0,X$) for $\rho_0\in\{0,X\}$ and call it \emph{admissible} if it is of the form $\rho=\rho_0(\tau_k)\dots(\tau_1)$ (including $0,X$) where each $\tau_i$ is small and $\rho_0\in\{0,X\}$ as before.

Further, we define certain subclasses of existential/universal formulas satisfying certain type restrictions: A formula $A$ is called a \emph{$\forall$-formula} if $A=\forall\underline{a}^{\underline{\rho}}A_{qf}(\underline{a})$ with $A_{qf}$ quantifier-free and all types $\rho_i$ in $\underline{\rho}=(\rho_1,\dots,\rho_k)$ are admissible. A formula $A$ is called an \emph{$\exists$-formula} if $A=\exists\underline{a}^{\underline{\rho}}A_{qf}(\underline{a})$ with similar $\underline{\rho}$.

\subsection{Treating operators via total resolvents}\label{labelz}

In this section, we extend the system $\mathcal{A}^\omega[X,\norm{\cdot}]$ from Section \ref{fliere} by constants and axioms with the following essential properties:
\begin{itemize}
\item We can formalize theorems and proofs involving abstract accretive and (generalized) monotone set-valued operators and their (total) resolvents.  
\item We obtain bound extraction results in the sense of the usual metatheorems of proof mining. 
\end{itemize}
The second item essentially amounts to whether the new constants can be majorized in a suitable sense by functionals of finite type (see \cite{GeK2008,Koh2005}) and whether the corresponding axioms have a monotone functional interpretation (see \cite{Koh1996}). The latter in particular is guaranteed for purely universal axioms.

\smallskip

We divide the corresponding presentation on whether the resolvents are assumed to be total or partial. Based on the correspondence results discussed in the previous chapter, totality of the resolvent is tied to maximality conditions of the operators and we will in particular be able to treat maximal (generalized) monotone operators by treating operators with total resolvents (see in particular Remark \ref{rem:maxremark} for a comparison with treating maximality in an ad hoc way). The following section now presents the underlying reasons for the particular choice of constants and axioms made later.

\subsubsection{Formal systems for total resolvents}

At first, to model a set-valued operator $A:X\to 2^X$ via functionals of finite type (in the sense of $T^X$), we add a constant $\chi_A$ of type $0(X)(X)$ to the language of normed spaces which represents $A$ using a function taking an argument $x$ from $X$ and returning a characteristic function for $Ax$\footnote{This is conceptually similar to the representation of a designated convex set $C$ in the systems $\mathcal{A}^\omega[X,\norm{\cdot},C]$ / $\mathcal{A}^\omega[X,\norm{\cdot},C]_{-b}$ from \cite{GeK2008,Koh2005}.}. In that vein, we write $y\in Ax$ for $\chi_Axy=_00$. 

For the resolvent, recall the definition given before:
\[
J^A_\gamma :=(Id+\gamma A)^{-1}.
\]
In that way, $J^A_\gamma$ is also a set-valued operator. However, as the previously discussed correspondence results between $A$ and $J^A_\gamma$ show, accretivity or (generalized) monotonicity of $A$ imply that $J^A_\gamma x$ is actually a singleton for any $x$ and is thus usually identified with a proper (potentially partial) function $X\to X$. Therefore, we can add a new constant $J^{\chi_A}$ of type $X(X)(1)$, taking both the parameter $\gamma$ and the argument of $J^A_\gamma$ and outputting the \emph{unique} return-value which exists for any input as forced by the type. Motivated by this, we write $J^A_\gamma$ for $J^{\chi_A}\gamma$ (which is an object of type $X(X)$). Any semantic interpretation of this language will thus have to interpret $J^{\chi_A}$ by some total function which, together with a suitable resolvent axiom expressing $J^A_\gamma :=(Id+\gamma A)^{-1}$, forces $A$ to be semantically interpreted by a maximal operator\footnote{This does not mean that maximality is provable in the to be defined systems as will be extensively discussed later on.}.

Now, for the right axiom expressing $J^A_\gamma=(Id+\gamma A)^{-1}$, consider this defining equality in the following formulation for $J^A_\gamma$ as a set-valued mapping:
\[
\forall \gamma^1,p^X,x^X\left(\gamma>_\mathbb{R}0_\mathbb{R}\rightarrow (p\in J^A_{\gamma} x\leftrightarrow\gamma^{-1}(x-_Xp)\in Ap)\right).\tag{$\dagger$}
\]
A natural step may be to replace $p\in J^A_\gamma x$ by $p=_X J^A_\gamma x$, under which the above ($\dagger$) transforms into
\[
\forall \gamma^1,p^X,x^X\left(\gamma>_\mathbb{R}0\rightarrow (p=_X J^A_{\gamma} x\leftrightarrow\gamma^{-1}(x-_Xp)\in Ap)\right)
\]
in the language of the constants $\chi_A$ and $J^{\chi_A}$. This, by the hidden universal quantifier in $p=_X J^A_\gamma x$, is of course not universal. By separating the two directions of the biimplication into
\[
\begin{cases}
\forall \gamma^1,p^X,x^X\left(\gamma>_\mathbb{R}0\land \gamma^{-1}(x-_Xp)\in Ap\rightarrow p=_X J^A_{\gamma} x\right),\\
\forall \gamma^1,p^X,x^X\left(\gamma>_\mathbb{R}0\land p=_X J^A_{\gamma}x\rightarrow\gamma^{-1}(x-_Xp)\in Ap)\right),
\end{cases}
\]
we see that the first one is unproblematic as it is universal but we can even ignore it for now as it will turn out to be provable in the systems later defined. For the latter, we remove the universal premise by weakening the statement to its intensional version:
\[
\forall \gamma^1,x^X\left(\gamma>_\mathbb{R}0\rightarrow \gamma^{-1}(x-_XJ^A_{\gamma}x)\in A(J^A_{\gamma}x)\right).
\]
This is in particular universal and will be right axiom to treat resolvents of set-valued operators. As we will later see, the strong version of the above resolvent axiom, i.e. 
\[
\forall\gamma^1,x^X,p^X\left(\gamma>_\mathbb{R}0\land p=_X J^A_{\gamma}x\rightarrow\gamma^{-1}(x-_Xp)\in Ap\right)
\]
turns out to have a strong connection to the maximality and extensionality statement for $A$.
\begin{remark}\label{rem:maxremark}
Instead of using the detour of total resolvents, one might consider expressing maximality of an operator more directly. Consider a monotone operator (where one has, compared to $m$-accretive operators, a true maximality statement). In the language of normed spaces extended with $\chi_A$, the monotonicity of $A$ can be swiftly expressed by the universal sentence 
\[
\forall x^X,y^X,u^X,v^X\left(u\in Ax\land v\in Ay\rightarrow\langle x-_Xy,u-_Xv\rangle_X\geq_\mathbb{R} 0\right).
\]
but formalizing maximal monotonicity naturally leads to the axiom
\[
\forall x^X,u^X\left(\forall y^X,v^X\left(v\in Ay\rightarrow\langle x-_Xy,u-_Xv\rangle_X\geq_\mathbb{R}0\right)\rightarrow u\in Ax\right)
\]
which is problematic for establishing such a bound-extraction result as it does not have a monotone functional interpretation.

Indeed, we will later formally see that this can not be avoided and that this route of treating maximal operators via total resolvents is in some sense optimal if one does not assume some uniform notion of continuity of $A$ since the maximality statement will turn out to be provably equivalent to extensionality of the operator $A$ which has to be unprovable in systems which deal with non-continuous operators and allow for bound extractions.
\end{remark}
With this motivation, we now turn to the precise definitions and the discussion of the systems treating the various kinds of operators considered before. Note for this, that the upcoming formal systems are concerned with only one space $X$ and one set-valued operator $A:X\to 2^X$. For extensions of this, in various directions, see the short discussion in Remark \ref{rem:extensionsMeta}.

\subsubsection{Logical systems for m-accretive operators}
We first introduce logical systems that can accommodate m-accretive operators and then, as a `litmus test', derive the basic properties of the latter in the former.

To define the system for $m$-accretive operators, or accretive operators with total resolvents, we add the constant $\chi_A$ for the set-valued operator and $J^{\chi_A}$ for the resolvent as before and, besides those, we add three further constants to the language of $\mathcal{A}^\omega[X,\norm{\cdot}]$, namely a constant $\tilde{\gamma}$ of type $1$, a constant $m_{\tilde{\gamma}}$ of type $0$ and a constant $c_X$ of type $X$. These are used for majorization of the resolvent constant $J^{\chi_A}$ later on in the sense that a bound for $\norm{x-J^A_\gamma x}$ for \emph{some} $x$ and \emph{some} $\gamma>0$ will suffice for constructing a majorant of $J^{\chi_A}$. In that way, $c_X$ designates such an arbitrary point and $\tilde{\gamma}$ such an arbitrary index where we use $m_{\tilde{\gamma}}$ to provide a quantitative version of $\tilde{\gamma}>0$ through stipulating $\tilde{\gamma}\geq_\mathbb{R}2^{-m_{\tilde{\gamma}}}$ in the following axioms.

The theory {$\mathcal{A}^\omega[X,\norm{\cdot},A,J^A]$\label{th:TotAcc}} is now officially defined as the extension of the theory $\mathcal{A}^\omega[X,\norm{\cdot}]$ with the above constants and corresponding axioms
\begin{enumerate}
\item[(I)] $\forall x^X,y^X(\chi_A xy\leq_01)$,
\item[(II)] $\forall \gamma^1,x^X\left(\gamma>_\mathbb{R}0\rightarrow \gamma^{-1}(x-_XJ^A_{\gamma}x)\in A(J^A_{\gamma}x)\right)$,
\item[(III)] $\begin{cases}\forall x^X,y^X,u^X,v^X,\lambda^1\big(u\in Ax\land v\in Ay\\
\qquad\qquad\qquad\rightarrow \norm{x-_Xy+_X\vert\lambda\vert(u-_Xv)}_X\geq_\mathbb{R}\norm{x-_Xy}_X\big),\end{cases}$
\item[(IV)] $\tilde{\gamma}\geq_\mathbb{R}2^{-m_{\tilde{\gamma}}}$.
\end{enumerate}
We use $\vert\lambda\vert$ to avoid the universal premise $\lambda\geq_\mathbb{R}0$ in axiom (III). Note that the behavior of $J^A_\gamma$ for $\gamma\leq_\mathbb{R}0$ is left undefined. 

The system $\mathcal{A}^\omega[X,\norm{\cdot},A,J^A]$ is strong enough to formalize large parts of the theory of $m$-accretive operators and we see examples of some essential theorems on the operator and resolvent that $\mathcal{A}^\omega[X,\norm{\cdot},A,J^A]$ proves in the following proposition (formalizing parts of Theorem \ref{thm:accretiveEquivalence}) as a first indication of that. We gives sketches of the formal proofs as they are quite instructive regarding the use of the basic axioms of the above system. For that, we write $\mathcal{V}^\omega$ as an abbreviation for $\mathcal{A}^\omega[X,\norm{\cdot},A,J^A]$.
\begin{proposition}\label{pro:respropaccr}
$\mathcal{V}^\omega$ proves:
\begin{enumerate}
\item[(1)] $J^A_\gamma$ is unique for any $\gamma>0$, i.e. 
\[
\forall \gamma^1,p^X,x^X\left(\gamma>_\mathbb{R}0\land \gamma^{-1}(x-_Xp)\in Ap\rightarrow p=_X J^A_{\gamma} x\right).
\]
\item[(2)] $J^A_\gamma$ is firmly nonexpansive for any $\gamma>0$, i.e.
\begin{align*}
&\forall\gamma^1,r^1,x^X,y^X\Big(\gamma>_\mathbb{R}0\land r>_\mathbb{R}0\rightarrow\norm{J^A_\gamma x-_XJ^A_\gamma y}_X\\
&\qquad\qquad\qquad\qquad\leq_\mathbb{R}\norm{r(x-_Xy)+(1-r)(J^A_\gamma x-_XJ^A_\gamma y)}_X\Big).
\end{align*}
\item[(3)] $J^A_\gamma$ is nonexpansive for any $\gamma>0$, i.e.
\[
\forall\gamma^1,x^X,y^X\left(\gamma>_\mathbb{R}0\rightarrow \norm{x-_Xy}_X\geq_\mathbb{R}\norm{J^A_\gamma x-_XJ^A_\gamma y}_X\right).
\]
\item[(4)] $J^{\chi_A}$ is extensional in both arguments:
\[
\forall \gamma^1>_\mathbb{R}0,{\gamma'}^1>_\mathbb{R}0,x^X,{x'}^X\left(x=_Xx'\land\gamma=_\mathbb{R}\gamma'\rightarrow J^A_\gamma x=_XJ^A_{\gamma'}x'\right).
\]
\end{enumerate}
\end{proposition}
\begin{proof}
\begin{enumerate}
\item[(1)] Suppose that $\gamma>0$ and that $\gamma^{-1}(x-p)\in Ap$. Axiom (II) gives $\gamma^{-1}(x- J^A_\gamma x)\in A(J^A_\gamma x)$. Axiom (III) then implies that
\[
0=\norm{p-J^A_\gamma x+\vert\gamma\vert(\gamma^{-1}(x-p)-\gamma^{-1}(x-J^A_\gamma x))}\geq\norm{p-J^A_\gamma x}.
\]
Thus $p=_XJ^A_\gamma x$ since provably $\gamma=\vert\gamma\vert$ assuming $\gamma>0$.
\item[(2)] Let $\gamma>0$. Axiom (II) gives 
\[
\gamma^{-1}(x-J^A_\gamma x)\in A(J^A_\gamma x)\text{ and }\gamma^{-1}(y-J^A_\gamma y)\in A(J^A_\gamma y).
\]
Axiom (III) gives
\begin{align*}
\norm{J^A_\gamma x-J^A_\gamma y}&\leq\norm{(J^A_\gamma x-J^A_\gamma y)+s(\gamma^{-1}(x-J^A_\gamma x) -\gamma^{-1}(y-J^A_\gamma y))}\\
&=\norm{(J^A_\gamma x-J^A_\gamma y)+s\gamma^{-1}((x-y)-(J^A_\gamma x-J^A_\gamma y))}\\
&=\norm{s\gamma^{-1}(x-y)+(1-s\gamma^{-1})(J^A_\gamma x-J^A_\gamma y)}
\end{align*}
for any $s>0$ using extensionality of the norm. By considering $s=r\gamma$, we get
\[
\norm{J^A_\gamma x-J^A_\gamma y}\leq\norm{r(x-y)+(1-r)(J^A_\gamma x-J^A_\gamma y)}
\]
for any $r>0$.
\item[(3)] Let $\gamma> 0$. By the previous item (2), we have for $r=1$:
\[
\forall\gamma^1,x^X,y^X\left(\gamma> 0\rightarrow\norm{J^A_\gamma x-J^A_\gamma y}\leq\norm{1(x-y)+(1-1)(J^A_\gamma x-J^A_\gamma y)}\right).
\]
We get nonexpansivity using basic arithmetic and extensionality of the norm in $\mathcal{A}^\omega[X,\norm{\cdot}]$.
\item[(4)] We get that $x=_Xx'$ implies $\norm{x-x'}=0$, i.e. $\norm{J^A_\gamma x-J^A_\gamma x'}=0$ by nonexpansivity. Thus $J^A_\gamma$ is extensional in $x$ for any $\gamma>0$. For extensionality in $\gamma$, let $\gamma=\gamma'$ be given with $\gamma,\gamma'>0$. Then $\gamma^{-1}=\gamma'^{-1}$. By axiom (II), we get that $\gamma^{-1}(x-J^A_{\gamma}x)\in A(J^A_{\gamma}x)$ and $\gamma'^{-1}(x-J^A_{\gamma'}x)\in A(J^A_{\gamma'}x)$. By axiom (III), we get 
\[
\norm{J^A_\gamma x-J^A_{\gamma'}x+\vert\gamma\vert(\gamma^{-1}(x-J^A_\gamma x)-\gamma'^{-1}(x-J^A_{\gamma'}x))}\geq\norm{J^A_\gamma x-J^A_{\gamma'}x}
\]
By extensionality of $\norm{\cdot}$ and the arithmetic operations in $X$, we get
\begin{align*}
0&=\norm{J^A_\gamma x-J^A_{\gamma'}x+((x-J^A_\gamma x)-(x-J^A_{\gamma'}x))}\\
&=\norm{J^A_\gamma x-J^A_{\gamma'}x+\vert\gamma\vert(\gamma^{-1}(x-J^A_\gamma x)-\gamma'^{-1}(x-J^A_{\gamma'}x))}\\
&\geq\norm{J^A_\gamma x-J^A_{\gamma'}x}.
\end{align*}
Thus $\norm{J^A_\gamma x-J^A_{\gamma'}x}=0$, i.e. $J^A_\gamma x=_XJ^A_{\gamma'}x$.
\end{enumerate}
\end{proof}
\begin{remark}\label{rem:division}
As discussed in the context of the representation of real numbers in $\mathrm{WE}$-$\mathrm{PA}^\omega$ already, some subtleties arise when dealing with reciprocals like in the above presented axioms and theorems and we want to indicate what these subtleties are and how they can be formally addressed. As mentioned in the discussion of real arithmetic, formulas containing reciprocal expressions like, e.g., the resolvent axiom
\[
\forall \gamma^1,x^X\left(\gamma>_\mathbb{R}0\rightarrow \gamma^{-1}(x-_XJ^A_{\gamma}x)\in A(J^A_{\gamma}x)\right)
\]
are just seen as abbreviations for extended versions which make the necessary dependency on a parameter $l^0$ with $\vert\gamma\vert >_\mathbb{R}2^{-l}$ explicit, i.e. in the above example, one actually considers
\[
\forall \gamma^1,x^X,l^0\left(\gamma >_\mathbb{R}2^{-l}\rightarrow (\gamma)^{-1}_l(x-_XJ^A_{\gamma}x)\in A(J^A_{\gamma}x)\right)
\]
where $(\cdot)^{-1}_l$ is the previously discussed closed term representing the reciprocal correctly for arguments $\alpha^1$ satisfying $\vert\alpha\vert>_\mathbb{R}2^{-l}$.

In most situations, like, e.g., in the formal theorems and proofs presented above, these details can be neglected without loss of generality. But they can feature prominently in some contexts, like the extraction of quantitative information. For example, item (4) of the above Proposition \ref{pro:respropaccr} established extensionality of the resolvent with respect to its first argument, the parameter $\gamma$, in $\mathcal{V}^\omega$, i.e.
\[
\mathcal{V}^\omega\vdash\forall x^X,\gamma^1,{\gamma'}^1\left(\gamma>_\mathbb{R}0\land\gamma'>_\mathbb{R}0\land\gamma=_\mathbb{R}\gamma'\rightarrow J^A_\gamma x=_XJ^A_{\gamma'}x\right).
\]
Making the hidden quantifiers apparent, we actually prove
\begin{align*}
&\forall x^X,\gamma^1,{\gamma'}^1,l^0,{l'}^0,k^0\exists j^0\Big(\gamma\geq_\mathbb{R}2^{-l}\land\gamma'\geq_\mathbb{R}2^{-l'}\\
&\qquad\qquad\qquad\qquad\land\vert\gamma-\gamma'\vert\leq_\mathbb{R}2^{-j}\rightarrow\norm{J^A_\gamma x-_XJ^A_{\gamma'}x}_X<_\mathbb{R}2^{-k}\Big).
\end{align*}
and a bound extraction result akin to the other metatheorems of proof mining (which we want to establish for $\mathcal{V}^\omega$) should guarantee a computable function realizing $j$ in terms of the parameters which is, moreover, highly uniform w.r.t. $x$ (maybe only depending on a weak upper bound on the norm) but which will depend in particular on $l$ and $l'$ besides the rather immediately obvious dependence on $k$ as an input! Such a function can indeed be obtained from the proof given above: By the resolvent axiom, we have
\[
\gamma^{-1}(x-J^A_\gamma x)\in A(J^A_\gamma x)\text{ and }\gamma'^{-1}(x-J^A_{\gamma'} x)\in A(J^A_{\gamma'} x)
\] 
and by accretivity of $A$, we get
\begin{align*}
\norm{J^A_\gamma x-J^A_{\gamma'}x}&\leq\norm{J^A_\gamma x-J^A_{\gamma'}x+\vert\gamma\vert(\gamma^{-1}(x-J^A_\gamma x)-\gamma'^{-1}(x-J^A_{\gamma'}x))}\\
&=\norm{\left(1-\frac{\gamma}{\gamma'}\right)x+\left(\frac{\gamma}{\gamma'}-1\right)J^A_{\gamma'}x}\\
&=\left\vert 1-\frac{\gamma}{\gamma'}\right\vert\norm{x-J^A_{\gamma'}x}
\end{align*}
as provably $\vert\gamma\vert=\gamma$ for $\gamma>0$.

Now, let $b\geq\norm{x-J^A_{\gamma'}x}$, $\gamma'\geq 2^{-l'}$ and $\vert\gamma-\gamma'\vert\leq 2^{-j}$. Then 
\[
\left\vert 1-\frac{\gamma}{\gamma'}\right\vert=\left\vert\frac{\gamma'-\gamma}{\gamma'}\right\vert\leq\frac{\vert\gamma'-\gamma\vert}{\gamma'}\leq\frac{2^{-j}}{2^{-l'}}.
\]
Thus, we have 
\[
\norm{J^A_\gamma x-J^A_{\gamma'}x}\leq b\frac{2^{-j}}{2^{-l'}}.
\]
Thus we may take $j=\floor{k+l'+\log_2b}$ as a realizer for $j$ (which is even independent from $l$). The dependence on $b$ is explained by the previously mentioned dependence of the upcoming metatheorems on a bound of the displacement $\norm{x-J^A_{\gamma}x}$ for arbitrary $x$ and $\gamma>0$ to majorize the resolvent.
\end{remark}
\begin{remark}\label{rem:yosidaremark}
Many other parts of the theory of $m$-accretive operators can be straightforwardly derived in $\mathcal{V}^\omega$. Just as an indication for that, we want to shortly mention some properties of a derived operator in this context: $A_\gamma$, the so-called \emph{Yosida approximate} (which is ubiquitous in the literature, see, e.g., \cite{Bar1976}), defined via
\[
A_\gamma:=\frac{1}{\gamma}(Id-J^A_\gamma).
\]
This can be treated by introducing $A_\gamma x$ as an abbreviation for the term $\frac{1}{\gamma}(x-J^A_\gamma x)$\footnote{Because of the subtleties with reciprocals, we do not define $A_\gamma$ via $\lambda$-abstraction.}. The following properties of the Yosida approximate are then provable in $\mathcal{V}^\omega$:
\begin{enumerate}
\item[(1)] $A_\gamma$ fulfills the characteristic inclusion for any $\gamma>0$, i.e.
\[
\forall\gamma^1,x^X\left(\gamma>_\mathbb{R}0\rightarrow A_\gamma x\in A(x-_X\gamma A_\gamma x)\right).
\]
\item[(2)] $A_\gamma$ is unique for any $\gamma>0$, i.e.
\[
\forall\gamma^1,p^X,x^X\left(\gamma>_\mathbb{R}0\land p\in A(x-_X\gamma p)\rightarrow p=_XA_\gamma x\right).
\]
\item[(3)] $A_\gamma$ is $2\gamma^{-1}$-Lipschitz continuous for any $\gamma>0$, i.e.
\[
\forall\gamma^1,x^X,y^X\left(\gamma>_\mathbb{R}0\rightarrow\norm{A_\gamma x-_XA_\gamma y}_X\leq_\mathbb{R}2\gamma^{-1}\norm{x-_Xy}_X\right).
\]
\item[(4)] $A_\gamma x$ is bounded by any $y\in Ax$ for any $\gamma>0$, i.e.
\[
\forall\gamma^1,x^X,y^X\left(\gamma>_\mathbb{R}0\land y\in Ax\rightarrow\norm{A_\gamma x}_X\leq_\mathbb{R}\norm{y}_X\right).
\]
\end{enumerate}
Proofs of those formalized theorems in particular rely on the use of $\Sigma_1\text{-}\mathrm{ER}$. 
\end{remark}
As discussed before, the bound extraction theorems established later yield that extensionality of $A$ can not be provable in our system as by using the former, we would obtain a quantitative interpretation of extensionality of $A$ which in this context amounts to a modulus of uniform continuity of $A$ (which will be formally discussed later on) which in general does not exist. We however still have the following weak rule of $A$-extensionality
\[
\frac{F\rightarrow s=_Xs'\quad F\rightarrow t=_Xt'}{F\rightarrow (s\in At\leftrightarrow s'\in At')}
\]
for an existential formula $F$ as a special case of the extensionality rule $\Sigma_1$-$\mathrm{ER}$ of $\mathcal{V}^\omega$.

Further, we will see that extensionality of $A$ is equivalent to the maximality statement for $A$ as well as to the previously mentioned stronger version of the resolvent axiom
\[
\forall \gamma^1,x^X,p^X\left(\gamma>_\mathbb{R}0\land p=_XJ^A_\gamma x\rightarrow \gamma^{-1}(x-_Xp)\in Ap\right).
\]
This is not in contradiction to Lemma \ref{lem:accretivechara} as it shows that extensionality of $A$ features in the respective proof in an essential way which thus can not be formalized in $\mathcal{V}^\omega$. In fact, the basis for the proof of parts of the upcoming theorem is that extensionality of $A$ is the only thing we need to add to $\mathcal{V}^\omega$ to formalize the proof of the lemma mentioned above.
\begin{theorem}\label{thm:extEquivAccretive}
Over $\mathcal{V}^\omega$, the following are equivalent:
\begin{enumerate}
\item[(1)] Extensionality of $A$, i.e.
\[
\forall x^X,y^X,{x'}^X,{y'}^{X}\left( x=_X x'\land y=_Xy'\rightarrow \chi_Axy=_0 \chi_Ax'y'\right).
\]
\item[(2)] The strong resolvent axiom, i.e.
\[
\forall x^X,p^X,\gamma^1\left(\gamma>_\mathbb{R}0\land p=_XJ^A_\gamma x\rightarrow \gamma^{-1}(x-_Xp)\in Ap\right).
\]
\item[(3)] Maximal accretivity of $A$, i.e.
\begin{align*}
&\forall x^X,u^X\Big(\forall y^X,v^X,\lambda^1\Big(v\in Ay\\
&\qquad\qquad\qquad\rightarrow\norm{x-_Xy+_X\vert\lambda\vert(u-_Xv)}_X\geq_\mathbb{R}\norm{x-_Xy}_X\Big)\rightarrow u\in Ax\Big).
\end{align*}
\item[(4)] The strong resolvent axiom for $\gamma=1$, i.e.
\[
\forall x^X,p^X\left(p=_XJ^A_{1} x\rightarrow (x-_Xp)\in Ap\right).
\]
\item[(5)] Maximal accretivity of $A$ for $\lambda=1$, i.e.
\begin{align*}
&\forall x^X,u^X\Big(\forall y^X,v^X\Big(v\in Ay\\
&\qquad\qquad\qquad\rightarrow\norm{x-_Xy+_X(u-_Xv)}_X\geq_\mathbb{R}\norm{x-_Xy}_X\Big)\rightarrow u\in Ax\Big).
\end{align*}
\item[(6)] Closure of the graph of $A$, i.e. 
\begin{align*}
\forall x^X,y^X,x_{(\cdot)}^{X(0)},y_{(\cdot)}^{X(0)}\Big(x_n\to_X x\land y_n\to_X y\land \forall n^0(y_n\in Ax_n)\rightarrow y\in Ax\Big)
\end{align*}
where $x_n\to_X x$ is short for
\[
\forall k^0\exists N^0\forall m\geq_0 N\left(\norm{x_m-_X x}_X\leq_\mathbb{R}2^{-k}\right)
\]
and similar for $y_n\to_X y$.
\end{enumerate}
\end{theorem}
\begin{proof}
For the whole proof, note that provably $1>0$ and thus $\vert 1\vert=1$.
\begin{enumerate}
\item [$(4)\Rightarrow (1)$] Let $x=_Xx'$ and $y=_Xy'$. Using (4), we get in particular that
\[
p=_XJ^A_1 x\leftrightarrow (x-p)\in Ap
\]
for all $p$. Therefore, we have
\begin{align*}
((y+x)-x)\in A x&\leftrightarrow x=_XJ^A_\gamma(y+x)\\
&\leftrightarrow x'=_XJ^A_\gamma(y'+x')\\
&\leftrightarrow ((y'+x')-x')\in Ax'
\end{align*}
using extensionality of $J^A_\gamma$. Now, provably without any assumptions, we have
\[
((y+x)-x)=_X y\text{ and }((y'+x')-x')=_Xy'.
\]
By the quantifier-free extensionality rule, we have $y\in A x\leftrightarrow y'\in Ax'$.
\item [$(1)\Rightarrow (5)$] Let $x,u$ be such that
\[
\forall y,v(v\in Ay\rightarrow\norm{x-y+(u-v)}\geq\norm{x-y}).
\]
Provably, without any assumptions, we have
\[
x+u=_XJ^A_1(x+u)+((x+u)-J^A_1(x+u))
\]
and by axiom (II), we have
\[
1^{-1}((x+u)-J^A_1(x+u))\in A(J^A_1(x+u))
\]
We get by assumption that
\begin{align*}
0&=\norm{x-J^A_1(x+u)+(u-1^{-1}((x+u)-J^A_1(x+u)))}\\
&\geq\norm{x-J^A_1(x+u)}
\end{align*}
using additionally the extensionality of the norm and the arithmetical operations. Thus $x=_XJ^A_1(x+u)$ and therefore $u=_Xx+u-J^A_1(x+u)=_X1^{-1}((x+u)-J^A_1(x+u))$. Thus, $1^{-1}((x+u)-J^A_1(x+u))\in A(J^A_1(x+u))$ implies $u\in Ax$ by extensionality of $A$.
\item [$(3)\Rightarrow (2)$] Let $\gamma\neq 0$ be given and assume $p=_XJ^A_\gamma x$. Then by axiom (II), we get $\gamma^{-1}(x-J^A_\gamma x)\in A(J^A_\gamma x)$. By accretivity of $A$, we get
\[
\forall y,v,\lambda(v\in Ay\rightarrow \norm{J^A_\gamma x-y+\vert\lambda\vert(\gamma^{-1}(x-J^A_\gamma x)-v)}\geq\norm{J^A_\gamma x-y}).
\]
Using extensionality of the norm, we get
\[
\forall y,v,\lambda(v\in Ay\rightarrow \norm{p-y+\vert\lambda\vert(\gamma^{-1}(x-p)-v)}\geq\norm{p-y}).
\]
By maximality, we get $\gamma^{-1}(x-p)\in Ap$.
\item [$(2)\Rightarrow (4)$ and $(5)\Rightarrow (3)$] Clear by using the extensionality rule for $A$.

\item [$(6)\Rightarrow (1)$] For $x=x'$ and $y=y'$, we have $(x)_n\to x'$ and $(y)_n\to y'$ where $(x)_n$ and $(y)_n$ are the constant $x$- and $y$-sequences, respectively. If $y\in Ax$, then by closure of the graph $y'\in Ax'$.

\item [$(5)\Rightarrow (6)$] Let $x_n\to x$ and $y_n\to y$ as well as $y_n\in Ax_n$ for all $n$. Let $v,w$ be arbitrary with $v\in Aw$. Then, by axiom (III)
\[
\norm{x_n-w+y_n-v}\geq\norm{x_n-w}
\]
for all $n$ and thus by taking the limit $\norm{x-w+y-v}\geq\norm{x-w}$. By maximal accretivity, as $v,w$ are arbitrary, we have $y\in Ax$.
\end{enumerate}
\end{proof}
It should be noted that this is not a constructed phenomenon resulting from tying the (maybe for a pure mathematician obscure) concept of extensionality of $A$ to the rather artificial concept of maximal accretiveness from Lemma \ref{lem:accretivechara} but that this also happens in the case of (generalized) monotone operators on Hilbert spaces where maximal accretiveness is replaced with maximal monotonicity or maximal $\rho$-comonotonicity which, in that case, are classically (meaning with extensionality) as strong as totality of the resolvent and thus are true maximality principles. Therefore, the link between maximality properties and extensionality of the operator $A$ seems to be rather strong (which is further supported by the fact that the proofs for the equivalences actually only use very weak fragments of $\mathcal{V}^\omega$).

In that way, this accretive case is particularly interesting as we have that a classically (meaning with extensionality of $A$) weaker property becomes incomparable in the absence of extensionality.
\begin{remark}\label{rem:optRem}
As mentioned before, the above results imply that maximality is not provable in $\mathcal{V}^\omega$ but we have used the classically (meaning with extensionality of $A$) stronger (or equivalent in the later cases of (generalized) monotone operators) principle of totality of the resolvents as an axiom in our system. This implies, in some sense, that the system $\mathcal{V}^\omega$ is optimal in strength among systems which allow for bound extraction results and do not require additional quantitative assumption on the operator $A$ like a uniform continuity assumption.
\end{remark}
\subsubsection{Logical systems for maximal monotone operators}
In this section, we introduce formal systems accommodating inner product spaces and corresponding maximal monotone operators (or monotone operators with total resolvents).

\smallskip

First of all, monotonicity and accretivity are equivalent for inner product spaces (see Theorem \ref{thm:monotoneEquivalence}), i.e.\ we can make (immediate) use of $\mathcal{V}^\omega$ . Hence, adding the axioms (I)-(IV) from before to $\mathcal{A}^\omega[X,\langle\cdot,\cdot\rangle]$ (or, in other words, adding the parallelogram law to $\mathcal{V}^\omega$) results in a corresponding system for monotone operators with total resolvents which we denote by {$\mathcal{A}^\omega[X,\langle\cdot,\cdot\rangle,A,J^A]$\label{th:TotMon}}.

We again begin with some elementary theorems of $\mathcal{T}^\omega=\mathcal{A}^\omega[X,\langle\cdot,\cdot\rangle,A,J^A]$ where, in particular, $\mathcal{T}^\omega$ will prove monotonicity of $A$ and thus behave as if we would have used this as an axiom instead (as accretivity is conversely provable from monotonicity in the inner product setting, see Remark \ref{rem:resolventprop}). In that vein, the following proposition formalizes parts of Theorem \ref{thm:monotoneEquivalence} and also in essence provides a proof for parts of Remark \ref{rem:altDefNE} regarding the alternative definition of firm nonexpansivity.

\begin{proposition}\label{pro:accImplMono}
$\mathcal{A}^\omega[X,\langle\cdot,\cdot\rangle]$ proves:
\begin{enumerate}
\item[(1)] $\forall x^X,y^X(\langle x,y\rangle_X\leq_\mathbb{R} 0\leftrightarrow \forall\alpha^1(\norm{x}_X\leq_\mathbb{R}\norm{x-_X\vert\alpha\vert y}_X))$.
\end{enumerate}
Further, $\mathcal{T}^\omega$ proves:
\begin{enumerate}
\item[(2)] $A$ is monotone, i.e.
\[
\forall x^X,y^X,u^X,v^X\left(u\in Ax\land v\in Ay\rightarrow\langle x-_Xy,u-_Xv\rangle_X\geq_\mathbb{R}0\right).
\]
\item[(3)] $J^A_\gamma$ satisfies the alternative notion of firm nonexpansivity for any $\gamma>0$, i.e.
\[
\forall\gamma^1,x^X,y^X\left(\gamma>_\mathbb{R}0\rightarrow \langle x-_Xy,J^A_\gamma x-_XJ^A_\gamma y\rangle_X\geq_\mathbb{R}\norm{J^A_\gamma x-_XJ^A_\gamma y}^2_X\right).
\]
\end{enumerate}
\end{proposition}
\begin{proof}
\begin{enumerate}
\item[(1)] A proof as hinted on in \cite{BC2017}, Lemma 2.13, (i) can be immediately formalized in the system and we thus omit the details.
\item[(2)] Let $u\in Ax$ and $v\in Ay$. By accretivity, for any $\lambda$, we get
\[
\norm{x-y+\vert\lambda\vert(u-v)}\geq\norm{x-y}.
\]
By item (1), we get
\[
0\geq\langle x-y,-(u-v)\rangle=-\langle x-y,u-v\rangle
\]
and thus $\langle x-y,u-v\rangle\geq 0$.
\item[(3)] We give a similar proof as in \cite{BC2017}, Proposition 4.4. Let $\gamma>0$. Applying item (1) to $J^A_\gamma x-J^A_\gamma y$ and $(J^A_\gamma x-J^A_\gamma y)-(x-y)$, we get
\begin{align*}
&\norm{J^A_\gamma x-J^A_\gamma y}^2\leq\langle x-y,J^A_\gamma x-J^A_\gamma y\rangle\\
&\quad\text{ iff }\langle J^A_\gamma x-J^A_\gamma y,(J^A_\gamma x -J^A_\gamma y)-(x-y)\rangle\leq 0\\
&\quad\text{ iff }\forall\alpha>0\left(\norm{J^A_\gamma x-J^A_\gamma y}\leq\norm{(J^A_\gamma x-J^A_\gamma y)-\alpha((J^A_\gamma x-J^A_\gamma y)-(x-y))}\right)\\
&\quad\text{ iff }\forall\alpha>0\left(\norm{J^A_\gamma x-J^A_\gamma y}\leq\norm{(1-\alpha)(J^A_\gamma x-J^A_\gamma y)+\alpha(x-y)}\right).
\end{align*}
The latter, i.e. the original notion of firm nonexpansivity, is provable by Proposition \ref{pro:respropaccr}.
\end{enumerate}
\end{proof}
We can obtain a similar characterization of extensionality of $A$ which, in this inner product case, is additionally equivalent to maximal monotonicity of $A$, as indicated by parts of Theorem \ref{thm:monotoneEquivalence}.
\begin{theorem}
Over $\mathcal{T}^\omega$, items (1) - (6) of Theorem \ref{thm:extEquivAccretive} are pairwise equivalent to each other and additionally to
\begin{enumerate}
\item[(7)] maximal monotonicity of $A$, i.e.
\[
\forall x^X,u^X\left(\forall y^X,v^X\left(v\in Ay\rightarrow\langle x-_Xy,u-_Xv\rangle_X\geq_\mathbb{R}0\right)\rightarrow u\in Ax\right).
\]
\end{enumerate}
\end{theorem}
\begin{proof}
By Theorem \ref{thm:extEquivAccretive}, it thus suffices to show $(3)\Leftrightarrow (7)$.
\begin{enumerate}
\item [$(3)\Rightarrow (7)$] Let $x,u$ be such that
\[
\forall y,v(v\in Ay\rightarrow \langle x-y,u-v\rangle\geq 0).
\]
By Proposition \ref{pro:accImplMono}, (1), we get
\[
\forall y,v(v\in Ay\rightarrow \forall\alpha(\norm{x-y+\vert\alpha\vert(u-v)}\geq\norm{x-y})).
\]
Using the assumed (3), we get $u\in Ax$.
\item [$(7)\Rightarrow (3)$] Let $x,u$ be such that
\[
\forall y,v,\lambda(v\in Ay\rightarrow\norm{x-y+\vert\lambda\vert(u-v)}\geq\norm{x-y}).
\]
Again, by Proposition \ref{pro:accImplMono}, (1), we get 
\[
\forall y,v(v\in Ay\rightarrow\langle x-y,u-v\rangle\geq 0)
\]
and thus, the assumed (7) implies $u\in Ax$.
\end{enumerate}
\end{proof}
In this light, maximality links with extensionality which is, as discussed before, in general unprovable (note that the previous Remark \ref{rem:optRem} also applies here). This potentially hinders formalizations of theorems that use said maximality which are ubiquitous in the literature. We will later see quantitative forms of extensionality which can be added to these systems to treat proofs where this features as an essential assumption. However, many applications actually only assume maximal monotonicity in order to use a total resolvent (or can be rephrased as such). This, together with the fact that the resolvent is provably extensional and that we still have the extensionality rule with existential premises, allows for substantial results on maximal monotone operators to be carried out in $\mathcal{T}^\omega$. For an example of this, we consider the following proposition. The original result (see, e.g., Proposition 23.31 from \cite{BC2017}) assumes maximal monotonicity of $A$. We will however see that a total extensional resolvent with an extensionality rule for $A$ is sufficient. Also, the properties will later be useful for the proof of the metatheorems.
\begin{proposition}\label{pro:resolventprop}
$\mathcal{T}^\omega$ proves:
\begin{enumerate}
\item[(1)] $\forall \gamma^1,\lambda^1,x^X\left(\gamma>_\mathbb{R}0\land\lambda>_\mathbb{R}0\rightarrow J^A_\lambda x=_X J^A_{\gamma}\left(\frac{\gamma}{\lambda} x+_X(1-\frac{\gamma}{\lambda})J^A_\lambda x\right)\right).$
\item[(2)] $\forall \gamma^1,\lambda^1,x^X\left(\gamma>_\mathbb{R}0\land\lambda>_\mathbb{R}0\rightarrow\norm{x-_XJ^A_\gamma x}_X\leq_\mathbb{R}\left(2+\frac{\gamma}{\lambda}\right)\norm{x-_XJ^A_\lambda x}_X\right).$
\end{enumerate}
\end{proposition}
\begin{proof}
\begin{enumerate}
\item[(1)] By axiom (II), we get
\[
\lambda^{-1}(x-J^A_\lambda x)\in A(J^A_\lambda x).
\]
Now, provably in $\mathcal{T}^\omega$ using the assumptions $\lambda>0$ and $\gamma>0$, we have
\[
\lambda^{-1}(x-J^A_\lambda x)=_X\frac{\gamma}{\lambda\gamma}(x-J^A_\gamma x)=_X\frac{1}{\gamma}\left(\frac{\gamma}{\lambda} x+\left(1-\frac{\gamma}{\lambda}\right)J^A_\lambda x-J^A_\lambda x\right).
\]
By the extensionality rule for $A$ with existential premises, we get
\[
\frac{1}{\gamma}\left(\frac{\gamma}{\lambda} x+\left(1-\frac{\gamma}{\lambda}\right)J^A_\lambda x-J^A_\lambda x\right)\in A(J^A_\lambda x)
\]
which implies $J^A_\lambda x=_XJ^A_{\gamma}\left(\frac{\gamma}{\lambda} x+\left(1-\frac{\gamma}{\lambda}\right)J^A_\lambda x\right)$ by Proposition \ref{pro:respropaccr}, (1).
\item[(2)] Using item (1) and Proposition \ref{pro:respropaccr}, (3), we get
\begin{align*}
\norm{x-J^A_\gamma x}&\leq\norm{J^A_\gamma x -J^A_\lambda x}+\norm{x-J^A_\lambda x}\\
&=\norm{J^A_\gamma x - J^A_{\gamma}\left(\frac{\gamma}{\lambda}x+\left(1-\frac{\gamma}{\lambda}\right)J^A_\lambda x\right)}+\norm{x-J^A_\lambda x}\\
&\leq\norm{x - \frac{\gamma}{\lambda}x-\left(1-\frac{\gamma}{\lambda}\right)J^A_\lambda x}+\norm{x-J^A_\lambda x}\\
&\leq \left(1+\left\vert 1-\frac{\gamma}{\lambda}\right\vert\right)\norm{x-J^A_\lambda x}\\
&\leq \left(2+\frac{\gamma}{\lambda}\right)\norm{x-J^A_\lambda x}.
\end{align*}
\end{enumerate}
\end{proof}
\begin{remark}\label{rem:resolventprop}
\begin{enumerate}
\item[(1)] The above applications of the extensionality rule rely on the assumptions $\gamma>_\mathbb{R}0$ and $\lambda>_\mathbb{R}0$ which are not quantifier-free but existential. Thus also here, the derivable $\Sigma_1\text{-}\mathrm{ER}$ from Remark \ref{rem:strongExtRule} is crucial.
\item[(2)] The above results can, as apparent from the proof, already be established in $\mathcal{V}^\omega$. The result is however much more instructive if phrased for $\mathcal{T}^\omega$, since (as already highlighted in the comment preceding the proposition) Proposition 23.31 in \cite{BC2017} actually states the results with the assumption of maximal monotonicity of $A$ which, however, is not necessary here but only totality of the resolvent together with a weak rule of extensionality of $A$. 
\item[(3)] Further nice applications of the extensionality rule (with existential premises) are the following alternative axiomatizations of the theories $\mathcal{T}^\omega$ and $\mathcal{V}^\omega$:
\begin{enumerate}
\item In any variant of the theory $\mathcal{V}^\omega$ where the accretivity axiom for $A$ is replaced by
\begin{enumerate}
\item nonexpansivity for $J^A_\gamma$ for all $\gamma>0$ (and uniqueness of the resolvent), i.e. 
\[
\begin{cases}\forall x^X,y^X,\gamma^1\left(\gamma>_\mathbb{R}0\rightarrow \norm{x-_Xy}_X\geq_\mathbb{R}\norm{J^A_{\gamma} x-_XJ^A_{\gamma} y}_X\right),\\
\forall \gamma^1,p^X,x^X\left(\gamma>_\mathbb{R}0\land \gamma^{-1}(x-_Xp)\in Ap\rightarrow p=_X J^A_{\gamma} x\right),\end{cases}
\]
\end{enumerate}
\begin{enumerate}
\item [or ii.] firm nonexpansivity for $J^A_\gamma$ for all $\gamma>0$ (and uniqueness of the resolvent), i.e. 
\[
\begin{cases}\forall\gamma^1,r^1,x^X,y^X\Big(\gamma>_\mathbb{R}0\land r>_\mathbb{R}0_\mathbb{R}\rightarrow\norm{J^A_\gamma x-_XJ^A_\gamma y}_X\\
\qquad\qquad\qquad\leq_\mathbb{R}\norm{r(x-_Xy)+_X(1-r)(J^A_\gamma x-_XJ^A_\gamma y)}_X\Big),\\
\forall \gamma^1,p^X,x^X\left(\gamma>_\mathbb{R}0\land \gamma^{-1}(x-_Xp)\in Ap\rightarrow p=_X J^A_{\gamma} x\right),\end{cases}
\]
\item [or iii.] firm nonexpansivity for $J^A_1$ (and uniqueness of the resolvent), i.e.
\[
\begin{cases}\forall r^1,x^X,y^X\Big(r>_\mathbb{R}0_\mathbb{R}\rightarrow\norm{J^A_{1} x-_XJ^A_{1} y}_X\\
\qquad\qquad\qquad\leq_\mathbb{R}\norm{r(x-_Xy)+_X(1-r)(J^A_{1} x-_XJ^A_{1} y)}_X\Big),\\
\forall \gamma^1,p^X,x^X\left(\gamma>_\mathbb{R}0\land \gamma^{-1}(x-_Xp)\in Ap\rightarrow p=_X J^A_{\gamma} x\right),\end{cases}
\]
\end{enumerate}
one can actually prove the accretivity axiom. This formalizes parts of Theorem \ref{thm:accretiveEquivalence} but we omit the corresponding proofs as the ones provided in \cite{Bar1976,Tak2000} can be almost immediately formalized.
\item In the theory $\mathcal{T}^\omega$, the accretivity axiom could be similarly replaced with either
\begin{enumerate}
\item monotonicity  of $A$, i.e. 
\[
\forall x^X,y^X,u^X,v^X\left(u\in Ax\land v\in Ay\rightarrow\langle x-_Xy,u-_Xv\rangle_X\geq_\mathbb{R} 0_\mathbb{R}\right),
\]
\end{enumerate}
\begin{enumerate}
\item [or ii.] with the alternative notion of firm nonexpansivity for $J^A_1$ (together with uniqueness of the resolvent), i.e.
\[
\begin{cases}
\forall x^X,y^X\left(\langle x-_Xy,J^A_{1} x-_XJ^A_{1} y\rangle_X\geq_\mathbb{R}\norm{J^A_{1} x-_XJ^A_{1} y}^2_X\right),\\
\forall p^X,x^X\left((x-_Xp)\in Ap\rightarrow p=_X J^A_{1} x\right).
\end{cases}
\]
\end{enumerate}
This formalizes parts of Theorem \ref{thm:monotoneEquivalence} but, again, the proofs provided in the corresponding standard references \cite{Bar1976,BC2017,EB1992} can be immediately formalized which is why we omit further details.
\end{enumerate}
\end{enumerate}
\end{remark}
\subsubsection{Logical systems for maximal $\rho$-comonotone operators}
Lastly, we formulate logical systems that can accommodate maximal $\rho$-comonotone operators, the latter as introduced in \cite{CP2004}.

\smallskip

These $\rho$-comonotone operators are a special class of generalized monotone operators in that we relax the above monotonicity inequality to
\[
\langle x-y,u-v\rangle\geq\rho\norm{u-v}^2
\]
for $(x,u),(y,v)\in\mathrm{gra}A$ and for some (potentially negative) parameter $\rho\in\mathbb{R}$.

\smallskip

For treating operators which are maximally $\rho$-comonotone, we add constants $\tilde{\rho}$ of type $1$ representing a name for $\rho$ and $n_{\tilde{\gamma}}$ of type $0$ to the language of $\mathcal{A}^\omega[X,\langle\cdot,\cdot\rangle,A,J^A]$. Again, $\tilde{\gamma}$ serves as an anchor index to majorize the resolvent via bounding $\norm{c_X-_XJ^A_{\tilde{\gamma}}c_X}_X$. For that, besides the previous condition $\tilde{\gamma}\geq_\mathbb{R}2^{-m_{\tilde{\gamma}}}$, we need a further condition $\tilde{\rho}\geq_\mathbb{R}-\tilde{\gamma}+2^{-n_{\tilde{\gamma}}}$ to quantitatively express $\tilde{\rho}>-\tilde{\gamma}$ as we actually only consider resolvents $J^A_\gamma$ for such indices $\gamma>0$ (as discussed later).

\smallskip

The theory of a maximal $\rho$-comonotone operators {$\mathcal{A}^\omega[X,\langle\cdot,\cdot\rangle,A,J^A,\tilde{\rho}]$\label{th:TotCoMon}} is then defined by extending $\mathcal{A}^\omega[X,\langle\cdot,\cdot\rangle]$ with the axioms
\begin{enumerate}
\item[(I)] $\forall x^X,y^X(\chi_A xy\leq_01)$,
\item[(II)] $\forall \gamma^1,x^X\left(\gamma>_\mathbb{R}0\rightarrow\gamma^{-1}(x-_XJ^A_{\gamma}x)\in A(J^A_{\gamma}x)\right)$,
\item[(III)] $\forall x^X,y^X,u^X,v^X\left(u\in Ax\land v\in Ay\rightarrow \langle x-_Xy,u-_Xv\rangle_X\geq_\mathbb{R}\tilde{\rho}\norm{u-_Xv}_X^2\right)$,
\item[(IV)] $\tilde{\gamma}\geq_\mathbb{R}2^{-m_{\tilde{\gamma}}}$,
\item[(V)] $\tilde{\rho}\geq_\mathbb{R}-\tilde{\gamma}+ 2^{-n_{\tilde{\gamma}}}$.
\end{enumerate}
We will at first develop most of the theory for maximally $\rho$-comonotone operators over $\mathcal{A}^\omega[X,\langle\cdot,\cdot\rangle,A,J^A,\tilde{\rho}]$. However, in the context of the bound extraction theorems later on, we will restrict to a version {$\mathcal{A}^\omega[X,\langle\cdot,\cdot\rangle,A,J^A,\tilde{\rho}^*]$\label{th:TotCoMonAlt}} where we only specify the behavior of the resolvents $J^A_\gamma$ for $\rho>-\gamma/2$, i.e. we replace axiom $\mathrm{(II)}$ by
\[
\forall \gamma^1,x^X\left(\tilde{\rho}>_\mathbb{R}-\gamma/2\land \gamma>_\mathbb{R}0_\mathbb{R}\rightarrow \gamma^{-1}(x-_XJ^A_{\gamma}x)\in A(J^A_{\gamma}x)\right)\tag{$\mathrm{II_1}$}
\]
and axiom $\mathrm{(IV)}$ by
\[
\tilde{\rho}\geq_\mathbb{R}-\tilde{\gamma}/2+2^{-n_{\tilde{\gamma}}}.\tag{$\mathrm{IV_1}$}
\]
We mainly do this as the resolvents with indices satisfying $\rho>-\gamma/2$ behave nicely w.r.t. majorization in the proof of the bound extraction theorem later on. This has, however, not a big impact on the range of the theorems as this assumption is anyway made in current applications (see \cite{Koh2021b}). See also the later Remark \ref{rem:UstarProp} for how that effects the following propositions.

We write $\mathcal{U}^\omega$ as a shorthand for the system $\mathcal{A}^\omega[X,\langle\cdot,\cdot\rangle,A,J^A,\tilde{\rho}]$ and $\mathcal{U}^{*\omega}$ for $\mathcal{A}^\omega[X,\langle\cdot,\cdot\rangle,A,J^A,\tilde{\rho}^*]$. We again exhibit some of the range of the formalizable theory by presenting some elementary properties of the resolvent that we can prove in the previously introduced systems. The following proposition is in that vein a partial formalization of Theorem \ref{thm:comonotoneEquivalence}. This result mainly stems from the work \cite{BMW2020} and most proofs can just be immediately formalized (after careful inspection of what axioms and rules are necessary in the corresponding formal system). For that reason, we omit most of the proofs and only include particularly interesting or instructive examples.
\begin{proposition}\label{pro:UFundTheorems}
$\mathcal{A}^\omega[X,\langle\cdot,\cdot\rangle]$ proves:
\begin{enumerate}
\item[(1)] For any $\alpha^1\in (0,1]$ and any $x^X,y^X$:
\[
\alpha^2(\norm{x}_X^2-\norm{(1-\alpha^{-1})x+_X\alpha^{-1}y}_X^2)=\alpha(\norm{x}_X^2-\alpha^{-1}(1-\alpha)\norm{x-_Xy}_X^2-\norm{y}_X^2).
\]
\item[(2)] For any $\alpha^1$ and any $x^X,y^X$:
\[
\alpha^2\norm{x}_X^2-\norm{(\alpha-1)x+_Xy}_X^2=2\alpha\langle x-_Xy,y\rangle_X-(1-2\alpha)\norm{x-_Xy}_X^2.
\]
\end{enumerate}
Further, $\mathcal{U}^\omega$ proves:
\begin{enumerate}
\item[(3)] $J^A_\gamma$ is single-valued if $\tilde{\rho}>-\gamma$, i.e. 
\[
\forall p^X,x^X,\gamma^1\left(\gamma>_\mathbb{R}0\land \tilde{\rho}>_\mathbb{R}-\gamma\land\gamma^{-1}(x-_Xp)\in Ap\rightarrow p=_X J^A_{\gamma} x\right).
\]
\item[(4)] $J^{\chi_A}$ is extensional in both arguments if $\tilde{\rho}>-\gamma$, i.e.
\begin{align*}
&\forall\gamma^1,{\gamma'}^1,x^X,{x'}^X\big(\gamma>_\mathbb{R}0\land{\gamma'}>_\mathbb{R}0\land\tilde{\rho}>_\mathbb{R}-\gamma\\
&\qquad\qquad\qquad\qquad\land x=_Xx'\land\gamma=_\mathbb{R}\gamma'\rightarrow J^A_\gamma x=_XJ^A_{\gamma'}x'\big).
\end{align*}
\item[(5)] $J^A_\gamma$ satisfies the alternative notion of being $\alpha$-conically nonexpansive for $\alpha=\frac{1}{2(\tilde{\rho}/\gamma+1)}$ if $\tilde{\rho}>-\gamma$, i.e.
\begin{align*}
&\forall \gamma^1,x^X,y^X\Big(\gamma>_\mathbb{R}0\land\tilde{\rho}>_\mathbb{R}-\gamma\\
&\qquad\qquad\qquad\rightarrow2\alpha\langle J^A_{\gamma}x-_XJ^A_{\gamma}y,(x-_X J^A_{\gamma}x)-_X(y-_X J^A_{\gamma}y)\rangle_X\\
&\qquad\qquad\qquad\qquad\qquad\qquad \geq_\mathbb{R}\left(1-2\alpha\right)\norm{(x-_XJ^A_{\gamma}x)-_X(y-_XJ^A_{\gamma}y)}_X^2\Big).
\end{align*}
\item[(6)] $J^A_\gamma$ is $\alpha$-conically nonexpansive for $\alpha=\frac{1}{2(\tilde{\rho}/\gamma+1)}$ if $\tilde{\rho}>-\gamma$, i.e. $\alpha\in (0,\infty)$ and
\[
(1-\alpha^{-1})Id+\alpha^{-1}J^A_\gamma\text{ is nonexpansive}.
\]
\item[(7)] $J^A_\gamma$ is nonexpansive if $\tilde{\rho}\geq-\gamma/2$, i.e.
\[
\forall\gamma^1,x^X,y^X\left(\gamma>_\mathbb{R}0\land\tilde{\rho}\geq_\mathbb{R}-\gamma/2\rightarrow\norm{J^A_\gamma x-_XJ^A_\gamma y}_X\leq\norm{x-_Xy}_X\right).
\]
\item[(8)] $J^A_\gamma$ satisfies the alternative notion of being $\alpha$-averaged for $\alpha=\frac{1}{2(\tilde{\rho}/\gamma+1)}$ if $\tilde{\rho}\geq-\gamma/2$, i.e.
\begin{align*}
&\forall \gamma^1,x^X,y^X\Big(\gamma>_\mathbb{R}0\land\tilde{\rho}\geq_\mathbb{R}-\gamma/2\\
&\qquad\qquad\qquad\rightarrow \left(1-\alpha\right)\norm{(x-_XJ^A_\gamma x)-_X(y-_XJ^A_\gamma y)}^2_X\\
&\qquad\qquad\qquad\qquad\qquad\qquad\leq_\mathbb{R}\alpha\left(\norm{x-_Xy}^2_X-\norm{J^A_\gamma x-_XJ^A_\gamma y}^2_X\right)\Big).
\end{align*}
\item[(9)] $J^A_\gamma$ is $\alpha$-averaged for $\alpha=\frac{1}{2(\tilde{\rho}/\gamma+1)}$ if $\tilde{\rho}>-\gamma/2$, i.e. $\alpha\in (0,1)$ and
\[
(1-\alpha^{-1})Id+\alpha^{-1}J^A_\gamma\text{ is nonexpansive}.
\]
\end{enumerate}
\end{proposition}
\begin{proof}
\begin{enumerate}
\item[(1)] A proof along the lines hinted in Lemma 2.17, (i) from \cite{BC2017} can be straightforwardly formalized. 
\item[(3)] Suppose $\gamma>0$ and $\tilde{\rho}>-\gamma$ and let $\gamma^{-1}(x-p)\in Ap$. By axiom (II), we get $\gamma^{-1}(x-J^A_\gamma x)\in A(J^A_\gamma x)$ and by axiom (III), we get
\begin{align*}
-\gamma^{-1}\norm{p-J^A_\gamma x}^2&=\gamma^{-1}\langle J^A_\gamma x-p,p-J^A_\gamma x\rangle\\
&=\langle J^A_\gamma x-p,\gamma^{-1}(x-J^A_\gamma x)-\gamma^{-1}(x-p)\rangle\\
&\geq\tilde{\rho}\norm{\gamma^{-1}(x-J^A_\gamma x)-\gamma^{-1}(x-p)}^2\\
&=\tilde{\rho}/\gamma^2\norm{p-J^A_\gamma x}^2.
\end{align*}
As $\tilde{\rho}>-\gamma$, we get $\norm{p-J^A_\gamma x}=0$, i.e. $p=_XJ^A\gamma x$.
\item[(4)] Let $\gamma,\gamma'>0$ with $\gamma=\gamma'$ and let $x=_Xx'$. By axiom (II), we get
\[
\gamma^{-1}(x-J^A_\gamma x)\in A(J^A_\gamma x)\text{ and }\gamma'^{-1}(x'-J^A_{\gamma'}x')\in A(J^A_{\gamma'}x').
\]
Axiom (III) gives
\begin{align*}
-\gamma^{-1}\norm{J^A_\gamma x-J^A_{\gamma'} x'}^2&=\gamma^{-1}\langle J^A_\gamma x-J^A_{\gamma'} x',J^A_{\gamma'} x'-J^A_\gamma x\rangle\\
&=\langle J^A_\gamma x-J^A_{\gamma'} x',\gamma^{-1}(x-J^A_\gamma x)-\gamma'^{-1}(x'-J^A_{\gamma'} x')\rangle\\
&\geq\tilde{\rho}\norm{\gamma^{-1}(x-J^A_\gamma x)-\gamma'^{-1}(x'-J^A_{\gamma'} x')}^2\\
&=\tilde{\rho}/\gamma^2\norm{J^A_{\gamma'} x'-J^A_\gamma x}^2
\end{align*}
where we in particular used the extensionality of the scalar product, the norm and the arithmetical operations on $X$ and $\mathbb{R}$. Again, as $\tilde{\rho}>-\gamma$, we get $J^A_{\gamma'} x'=_XJ^A_\gamma x$.
\end{enumerate}
\end{proof}
Note that the above items (6) and (9) are formalizations for parts of Remark \ref{rem:altDefNE} regarding the alternative definitions for averaged and conically nonexpansive mappings.
\begin{remark}\label{rem:yosidaremarkcomon}
Similarly to Remark \ref{rem:yosidaremark}, a much larger part of the theory of $\rho$-comonotone operators can be formalized rather immediately and we again exemplify this by noting some essential properties of the important Yosida approximate which are provable in $\mathcal{U}^\omega$, namely
\begin{enumerate}
\item[(1)] $A_\gamma$ fulfills the characteristic inclusion for any $\gamma>0$, i.e.
\[
\forall\gamma^1,x^X(\gamma>_\mathbb{R}0\rightarrow A_\gamma x\in A(x-_X\gamma A_\gamma x)).
\]
\item[(2)] $A_\gamma$ is unique for any $\gamma>0$ with $\tilde{\rho}>-\gamma$, i.e.
\[
\forall\gamma^1,p^X,x^X(\gamma>_\mathbb{R}0\land\tilde{\rho}>_\mathbb{R}-\gamma\land p\in A(x-_X\gamma p)\rightarrow p=_XA_\gamma x).
\]
\item[(3)] $A_\gamma$ is $2\gamma^{-1}$-Lipschitz continuous for any $\gamma>0$ with $\tilde{\rho}\geq-\gamma/2$, i.e.
\[
\forall\gamma^1,x^X,y^X(\gamma>_\mathbb{R}0\land\tilde{\rho}\geq_\mathbb{R}-\gamma/2\rightarrow\norm{A_\gamma x-_XA_\gamma y}_X\leq_\mathbb{R}2\gamma^{-1}\norm{x-_Xy}_X).
\]
\item[(4)] $A_\gamma x$ is bounded by any $y\in Ax$ for any $\gamma>0$ with $\tilde{\rho}\geq-\gamma/2$, i.e.
\[
\forall\gamma^1,x^X,y^X(\gamma>_\mathbb{R}0\land\tilde{\rho}\geq_\mathbb{R}-\gamma/2\land y\in Ax\rightarrow\norm{A_\gamma x}_X\leq_\mathbb{R}\norm{y}_X).
\]
\end{enumerate}
\end{remark}
Also in the context of comonotone operators, we can get a nice characterization of maximal $\rho$-comonotonicity in terms of extensionality of the operator $A$, akin to the previous results. This amounts to formally carrying out the proof given in \cite{BMW2020} for parts of Theorem \ref{thm:comonotoneEquivalence}. 
\begin{theorem}\label{thm:comonExtEquiv}
Over $\mathcal{U}^\omega$, the following are equivalent:
\begin{enumerate}
\item[(1)] Extensionality of $A$, i.e.
\[
\forall x^X,y^X,{x'}^X,{y'}^{X}\left(x=_X x'\land y=_Xy'\rightarrow \chi_Axy=_0 \chi_Ax'y'\right)
\]
\item[(2)] The strong resolvent axiom for $\rho>-\gamma$, i.e.
\[
\forall \gamma^1,x^X,p^X\left(\gamma>_\mathbb{R}0\land\tilde{\rho}>_\mathbb{R}-\gamma\land p=_XJ^A_{\gamma}x \rightarrow \gamma^{-1}(x-_Xp)\in Ap\right).
\]
\item[(3)] Maximal $\rho$-comonotonicity of $A$, i.e.
\[
\forall x^X,u^X\big(\forall y^X,v^X\big(v\in Ay\rightarrow\langle x-_Xy,u-_Xv\rangle_X\geq_\mathbb{R}\tilde{\rho}\norm{u-_Xv}_X^2\big)\rightarrow u\in Ax\big).
\]
\item[(4)] Closure of the graph of $A$, i.e. 
\begin{align*}
\forall x^X,y^X,x_{(\cdot)}^{X(0)},y_{(\cdot)}^{X(0)}\Big(x_n\to_X x\land y_n\to_X y\land \forall n^0(y_n\in Ax_n)\rightarrow y\in Ax\Big).
\end{align*}
with $\rightarrow_X$ as before.
\end{enumerate}
\end{theorem}
\begin{proof}
\begin{enumerate}
\item [$(1)\Rightarrow (2)$] Let $\gamma>0$ and $\tilde{\rho}>-\gamma$ and suppose $p=_XJ^A_\gamma x$. By axiom (II), we get
\[
\gamma^{-1}(x-J^A_\gamma x)\in A(J^A_\gamma x).
\]
Extensionality of $A$ gives $\gamma^{-1}(x-p)\in Ap$.
\item [$(2)\Rightarrow (3)$] Let $x,u$ be such that
\[
\forall y,v(v\in Ay\rightarrow\langle x-y,u-v\rangle\geq\tilde{\rho}\norm{u-v}^2).
\]
Let $\gamma>0$ be such that $\tilde{\rho}>-\gamma$. We have
\[
\gamma^{-1}((x+\gamma u)-J^A_\gamma(x+\gamma u))\in A(J^A_\gamma(x+\gamma u))
\]
by axiom (II). Then, we get
\begin{align*}
&-\gamma^{-1}\norm{x-J^A_\gamma(x+\gamma u)}^2\\
&\qquad\qquad=\langle x-J^A_\gamma(x+\gamma u),-\gamma^{-1}(x-J^A_\gamma(x+\gamma u))\rangle\\
&\qquad\qquad=\langle x-J^A_\gamma(x+\gamma u),u-\gamma^{-1}((x+\gamma u)-J^A_\gamma(x+\gamma u))\rangle\\
&\qquad\qquad\geq\tilde{\rho}\norm{u-\gamma^{-1}((x+\gamma u)-J^A_\gamma(x+\gamma u))}^2\\
&\qquad\qquad=\tilde{\rho}\gamma^{-2}\norm{x-J^A_\gamma(x+\gamma u)}^2.
\end{align*}
As $\tilde{\rho}>-\gamma$, we get $\norm{x-J^A_\gamma(x+\gamma u)}^2=0$, i.e. $x=_XJ^A_\gamma(x+\gamma u)$. By assumption of (2), we have
\[
\gamma^{-1}((x+\gamma u)-x)\in Ax
\]
which implies $u\in Ax$ via the extensionality rule.
\item [$(3)\Rightarrow (4)$] Let $x_n\to x$ and $y_n\to_X y$ as well as $y_n\in Ax_n$ for all $n$. Let $v,w$ be arbitrary with $v\in Aw$. Then, by axiom (III)
\[
\langle x_n-w,y_n-v\rangle\geq\tilde{\rho}\norm{y_n-v}^2
\]
for all $n$ and thus by taking the limit $\langle x-w,y-v\rangle\geq\tilde{\rho}\norm{y-v}^2$. By maximal $\rho$-comonotonicity, as $v,w$ were arbitrary, we have $y\in Ax$.
\item [$(4)\Rightarrow (1)$] Similar to $(6)\Rightarrow (1)$ of Theorem \ref{thm:extEquivAccretive}.
\end{enumerate}
\end{proof}
By formalizing Proposition 2.4 of \cite{Koh2021b}, we get a similar result on the fundamental resolvent equality as given in Proposition \ref{pro:resolventprop}, now for the system $\mathcal{U}^\omega$ (modulo some requirements on $\tilde{\rho}$).
\begin{proposition}\label{pro:resolventpropU}
$\mathcal{U}^\omega$ proves:
\begin{enumerate}
\item[(1)] $\begin{cases}\forall \gamma^1,\lambda^1,x^X\big(\gamma>_\mathbb{R}0\land\lambda>_\mathbb{R}0\land\tilde{\rho}>_\mathbb{R}-\gamma\\
\qquad\qquad\qquad\rightarrow J^A_\lambda x=_X J^A_{\gamma}\left(\frac{\gamma}{\lambda} x+_X\left(1-\frac{\gamma}{\lambda}\right)J^A_\gamma x\right)\big).\end{cases}$
\item[(2)] $\begin{cases}\forall \gamma^1,\lambda^1,x^X\Big(\gamma>_\mathbb{R}0\land\lambda>_\mathbb{R}0\land\tilde{\rho}\geq_\mathbb{R}-\frac{\gamma}{2}\\
\qquad\qquad\qquad\rightarrow\norm{x-_XJ^A_\gamma x}_X\leq\left(2+\frac{\gamma}{\lambda}\right)\norm{x-_XJ^A_\lambda x}_X\Big).\end{cases}$
\end{enumerate}
\end{proposition}
\begin{proof}
For items (1) and (2), the proof is essentially the same as for Proposition \ref{pro:resolventprop}. One just replaces Proposition \ref{pro:respropaccr}, (1) by Proposition \ref{pro:UFundTheorems}, (3) and Proposition \ref{pro:respropaccr}, (3) by Proposition \ref{pro:UFundTheorems}, (7) respectively.
\end{proof}
\begin{remark}\label{rem:UstarProp}
The above propositions also hold for $\mathcal{U}^{*\omega}$ in a slightly modified form:
\begin{enumerate}
\item[(1)] Proposition \ref{pro:UFundTheorems}, items (3) - (6) hold with $\tilde{\rho}>_\mathbb{R}-\gamma/2$ instead of $\tilde{\rho}>_\mathbb{R}-\gamma$. Items (7) and (8) hold with $\tilde{\rho}>_\mathbb{R}-\gamma/2$ instead of $\tilde{\rho}\geq_\mathbb{R}-\gamma/2$. Item (9) stays valid unchanged.
\item[(2)] Theorem \ref{thm:comonExtEquiv} holds with $\tilde{\rho}>_\mathbb{R}-\gamma/2$ instead of $\tilde{\rho}>_\mathbb{R}-\gamma$ for the strong resolvent axiom.
\item[(3)] Proposition \ref{pro:resolventpropU}, item (1) holds with $\tilde{\rho}>_\mathbb{R}-\lambda/2\land\tilde{\rho}>_\mathbb{R}-\gamma/2$ instead of $\tilde{\rho}>_\mathbb{R}-\gamma$. Item (2) holds with $\tilde{\rho}>_\mathbb{R}-\lambda/2\land\tilde{\rho}>_\mathbb{R}-\gamma/2$ instead of $\tilde{\rho}\geq_\mathbb{R}-\gamma/2$.
\end{enumerate}
\end{remark}
\begin{remark}\label{rem:resolventpropU}
We again can consider alternative axiomatizations for the theories $\mathcal{U}^\omega$ and $\mathcal{U}^{*\omega}$:
\begin{enumerate}
\item[(1)] In any variant of the theory $\mathcal{U}^\omega$ where the $\rho$-comonotonicity axiom for $A$ is replaced by
\begin{enumerate}
\item[i.] $J^A_\gamma$ is $\alpha$-conically nonexpansive (and unique) for $\gamma>0$, $\tilde{\rho}>-\gamma$ with $\alpha=\frac{1}{2(\tilde{\rho}/\gamma+1)}$, i.e.
\begin{align*}
\begin{cases}
\forall \gamma^1,x^X,y^X\Big(\gamma>_\mathbb{R}0\land\tilde{\rho}>_\mathbb{R}-\gamma\\
\qquad\qquad\rightarrow2\alpha\langle J^A_{\gamma}x-_XJ^A_{\gamma}y,(x-_X J^A_{\gamma}x)-_X(y-_X J^A_{\gamma}y)\rangle_X\\
\qquad\qquad\qquad\qquad \geq_\mathbb{R}\left(1-2\alpha\right)\norm{(x-_XJ^A_{\gamma}x)-_X(y-_XJ^A_{\gamma}y)}_X^2\Big),\\
\forall p^X,x^X,\gamma^1\big(\gamma>_\mathbb{R}0\land \tilde{\rho}>_\mathbb{R}-\gamma\land \gamma^{-1}(x-_Xp)\in Ap\rightarrow p=_X J^A_{\gamma} x\big),
\end{cases}
\end{align*}
\item[or ii.] $J^A_\gamma$ is $\alpha$-averaged (and unique) for $\gamma>0$, $\tilde{\rho}>-\gamma/2$  with $\alpha=\frac{1}{2(\tilde{\rho}/\gamma+1)}$, i.e.
\begin{align*}
\begin{cases}
\forall \gamma^1,x^X,y^X\Big(\gamma>_\mathbb{R}0\land\tilde{\rho}>_\mathbb{R}-\gamma/2\\
\qquad\qquad\rightarrow \left(1-\alpha\right)\norm{(x-_XJ^A_\gamma x)-_X(y-_XJ^A_\gamma y)}^2_X\\
\qquad\qquad\qquad\qquad\leq_\mathbb{R}\alpha\left(\norm{x-_Xy}^2_X-\norm{J^A_\gamma x-_XJ^A_\gamma y}^2_X\right)\Big),\\
\forall p^X,x^X,\gamma^1\big(\gamma>_\mathbb{R}0\land \tilde{\rho}>_\mathbb{R}-\gamma/2\land \gamma^{-1}(x-_Xp)\in Ap\rightarrow p=_X J^A_{\gamma} x\big),
\end{cases}
\end{align*}
\end{enumerate}
one can actually prove the $\rho$-comonotonicity axiom. These equivalences also stem from the work \cite{BMW2020} and arguments along the lines discussed there can be immediately formalized which is why we omit further details.
\item[(2)] The same holds for the theory $\mathcal{U}^{*\omega}$ if in item (i), $\tilde{\rho}>_\mathbb{R}-\gamma$ is replaced by $\tilde{\rho}>_\mathbb{R}-\gamma/2$.
\end{enumerate}
This formalizes parts of Theorem \ref{thm:comonotoneEquivalence}.
\end{remark}
\subsection{Treating operators with partial resolvents}
Some applications of accretive or (generalized) monotone operators do not require full maximality but only impose certain so-called \emph{range conditions} on the operator which force the domains of the resolvents to be 'large enough' (which will be discussed in some detail later on). To accommodate for such operators, we now discuss how the previous approach needs to be modified to treat partial resolvents.

\smallskip

We opt for the following strategy: We still use a constant $J^{\chi_A}$ of type $X(X)(1)$. Instead of specifying the behavior of this constant on any point $x$ as, e.g., done by 
\[
\forall \gamma^1,x^X\left(\gamma>_\mathbb{R}0\rightarrow \gamma^{-1}(x-_XJ^A_{\gamma}x)\in A(J^A_{\gamma}x)\right),
\]
we only specify it on its domain in the sense of
\[
\forall \gamma^1,x^X\left(\gamma>_\mathbb{R}0\land x\in\mathrm{dom}J^A_\gamma\rightarrow \gamma ^{-1}(x-_XJ^A_{\gamma}x)\in A(J^A_{\gamma}x)\right).
\]
In aiming for bound extraction theorems, this is of course only a viable option if $x\in\mathrm{dom}J^A_\gamma$ can be suitably represented such that the resulting axiom has a monotone functional interpretation.

Total or not, the domain of the resolvent always fulfills, as discussed in Section \ref{sec:resolventprelim}, the equation
\[
\mathrm{dom}J^A_\gamma=\mathrm{ran}(Id+\gamma A)
\]
and the latter is definable via
\begin{align*}
x\in\mathrm{ran}(Id+\gamma A)&\text{ iff }\exists y\left(x\in y+\gamma Ay\right)\\
&\text{ iff }\exists y,z\left(z\in Ay\land x=y+\gamma z\right)\\
&\text{ iff }\exists y,z\left(z\in Ay\land z=\frac{1}{\gamma}(x-y)\right).
\end{align*}
Now, the hidden universal quantifier in $z=\frac{1}{\gamma}(x-y)$, if formalized via $=_X$, can't be majorized and thus this version can't be freely added as an implicative assumption. However, we can opt for the weaker intensional version
\[
\exists y^X\left(\frac{1}{\gamma}(x-_Xy)\in Ay\right).
\]
We will later see that the stronger version is connected with extensionality of $A$, similar as in the case of the previous strong resolvent axiom and thus has to be unprovable. Therefore, with that choice we obtain the theories\label{th:PartOp} 
\begin{enumerate}
\item[(1)] $\mathcal{A}^\omega[X,\norm{\cdot},A,J^A]_p$,
\item[(2)] $\mathcal{A}^\omega[X,\langle\cdot,\cdot\rangle,A,J^A]_p$,
\item[(3)] $\mathcal{A}^\omega[X,\langle\cdot,\cdot\rangle,A,J^A,\tilde{\rho}]_p$,
\item[(4)] $\mathcal{A}^\omega[X,\langle\cdot,\cdot\rangle,A,J^A,\tilde{\rho}^*]_p$,
\end{enumerate}
from the previous ones by replacing the axiom
\[
\forall \gamma^1,x^X\left(\gamma>_\mathbb{R}0\rightarrow\gamma^{-1}(x-_XJ^A_{\gamma}x)\in A(J^A_{\gamma}x)\right)\tag{II}
\]
from before with
\[
\forall \gamma^1,x^X\left(\gamma>_\mathbb{R}0\land \exists y^X\left(\gamma^{-1}(x-_Xy)\in Ay\right)\rightarrow \gamma^{-1}(x-_XJ^A_{\gamma}x)\in A(J^A_{\gamma}x)\right).\tag{II$'$}
\]
In the case of $\mathcal{A}^\omega[X,\langle\cdot,\cdot\rangle,A,J^A,\tilde{\rho}^*]_p$, we replace the previous
\[
\forall \gamma^1,x^X\left(\tilde{\rho}>_\mathbb{R}-\gamma/2\land\gamma>_\mathbb{R}0\rightarrow \gamma^{-1}(x-_XJ^A_{\gamma}x)\in A(J^A_{\gamma}x)\right)\tag{II$_1$}
\]
by
\begin{align*}
&\forall \gamma^1,x^X\big(\tilde{\rho}>_\mathbb{R}-\gamma/2\land\gamma>_\mathbb{R}0\\
&\qquad\qquad\qquad\land \exists y^X\left(\gamma^{-1}(x-_Xy)\in Ay\right)\rightarrow\gamma^{-1}(x-_XJ^A_{\gamma}x)\in A(J^A_{\gamma}x)\big).\tag{II$'_1$}
\end{align*}
The constant $c_X$, which was previously only used to designate an arbitrary anchor point for majorization, is now used to actually designate a common element of the domains of all $J^A_\gamma$ (with $\tilde{\rho}>-\gamma/2$ in the case of $\mathcal{A}^\omega[X,\langle\cdot,\cdot\rangle,A,J^A,\tilde{\rho}^*]_p$) and for that we add the corresponding defining axiom
\[
\forall\gamma^1\left(\gamma>_\mathbb{R}0\rightarrow\gamma^{-1}(c_X-_XJ^A_{\gamma}c_X)\in A(J^A_{\gamma}c_X)\right)\tag{V}
\]
which we vary, in the case of $\mathcal{A}^\omega[X,\langle\cdot,\cdot\rangle,A,J^A,\tilde{\rho}^*]_p$, to
\[
\forall\gamma^1\left(\gamma>_\mathbb{R}0\land\tilde{\rho}>_\mathbb{R}-\gamma/2\rightarrow \gamma^{-1}(c_X-_XJ^A_{\gamma}c_X)\in A(J^A_{\gamma}c_X)\right).\tag{V$_1$}
\]
This assumption is easily satisfiable in many applications as the operator $A$ is often assumed to have a non-empty domain and that it satisfies a range condition like
\[
\mathrm{dom} A\subseteq\bigcap_{\gamma>0}\mathrm{ran}(Id+\gamma A).
\]

We use the shorthands $\mathcal{V}^\omega_p$, $\mathcal{T}^\omega_p$, $\mathcal{U}^\omega_p$ and $\mathcal{U}^{*\omega}_p$ for the systems (1) - (4) and from now on use the abbreviation
\[
x\in\mathrm{dom}(J^A_\gamma):=\exists y^X\left(\gamma^{-1}(x-_Xy)\in Ay\right).
\]
We obtain the following proposition as an immediate generalization of the previous Propositions \ref{pro:respropaccr}, \ref{pro:accImplMono} and \ref{pro:UFundTheorems}.
\begin{proposition}\label{pro:partialFundProp}
$\mathcal{V}^\omega_p$ proves:
\begin{enumerate}
\item[(1)] $J^A_\gamma$ is unique for any $\gamma>0$, i.e. 
\[
\forall \gamma^1,p^X,x^X\left(\gamma>_\mathbb{R}0\land\gamma^{-1}(x-_Xp)\in Ap\rightarrow p=_X J^A_{\gamma} x\right).
\]
\item[(2)] $J^A_\gamma$ is firmly nonexpansive for any $\gamma>0$ (on its domain), i.e.
\begin{align*}
&\forall\gamma^1,r^1,x^X,y^X\Big(\gamma>_\mathbb{R}0\land x\in\mathrm{dom}(J^A_\gamma)\land y\in\mathrm{dom}(J^A_\gamma)\land r>_\mathbb{R}0\\
&\qquad\quad\rightarrow\norm{J^A_\gamma x-_XJ^A_\gamma y}_X\leq_\mathbb{R}\norm{r(x-_Xy)+_X(1-r)(J^A_\gamma x-_XJ^A_\gamma y)}_X\Big).
\end{align*}
\item[(3)] $J^A_\gamma$ is nonexpansive for any $\gamma>0$ (on its domain), i.e.
\begin{align*}
&\forall\gamma^1,x^X,y^X\Big(\gamma>_\mathbb{R}0\land x\in\mathrm{dom}(J^A_\gamma)\land y\in\mathrm{dom}(J^A_\gamma)\\
&\qquad\qquad\qquad\qquad\rightarrow \norm{x-_Xy}_X\geq_\mathbb{R}\norm{J^A_\gamma x-_XJ^A_\gamma y}_X\Big).
\end{align*}
\item[(4)] $J^A$ is extensional in both arguments (on its domain), i.e.
\begin{align*}
\begin{cases}
\forall \gamma^1>_\mathbb{R}0,x^X,{x'}^X\big(x\in\mathrm{dom}(J^A_\gamma)\\
\qquad\qquad\qquad\land x'\in\mathrm{dom}(J^A_\gamma)\land x=_Xx'\rightarrow J^A_\gamma x=_XJ^A_{\gamma'}x'\big),\\
\forall \gamma^1>_\mathbb{R}0,{\gamma'}^1>_\mathbb{R}0,x^X\big(x\in\mathrm{dom}(J^A_\gamma)\\
\qquad\qquad\qquad\land x\in\mathrm{dom}(J^A_{\gamma'})\land \gamma=_\mathbb{R}\gamma'\rightarrow J^A_\gamma x=_XJ^A_{\gamma'}x\big).
\end{cases}
\end{align*}
\end{enumerate}
Further, $\mathcal{T}^\omega_p$ proves:
\begin{enumerate}
\item[(5)] $A$ is monotone, i.e.
\[
\forall x^X,y^X,u^X,v^X\left(u\in Ax\land v\in Ay\rightarrow\langle x-_Xy,u-_Xv\rangle_X\geq_\mathbb{R} 0\right).
\]
\item[(6)] $J^A_\gamma$ satisfies the alternative notion of firm nonexpansivity for any $\gamma>0$ (on its domain), i.e.
\begin{align*}
&\forall\gamma^1,x^X,y^X\Big(\gamma>_\mathbb{R}0\land x\in\mathrm{dom}(J^A_\gamma)\land y\in\mathrm{dom}(J^A_\gamma)\\
&\qquad\qquad\qquad\rightarrow \langle x-_Xy,J^A_\gamma x-_XJ^A_\gamma y\rangle_X\geq_\mathbb{R}\norm{J^A_\gamma x-_XJ^A_\gamma y}^2_X\Big).
\end{align*}
\end{enumerate}
Lastly, $\mathcal{U}^\omega_p$ proves:
\begin{enumerate}
\item[(7)] $J^A_\gamma$ is single-valued if $\tilde{\rho}>-\gamma$, i.e. 
\[
\forall p^X,x^X,\gamma^1\left(\gamma>_\mathbb{R}0\land \tilde{\rho}>_\mathbb{R}-\gamma\land\gamma^{-1}(x-_Xp)\in Ap\rightarrow p=_X J^A_{\gamma} x\right).
\]
\item[(8)] $J^A$ is extensional in both arguments (on its domain) if $\tilde{\rho}>-\gamma$, i.e. 
\[
\begin{cases}
\forall \gamma^1>_\mathbb{R}0,x^X,{x'}^X\big(\tilde{\rho}>_\mathbb{R}-\gamma\land x\in\mathrm{dom}(J^A_\gamma)\\
\qquad\qquad\qquad\land x'\in\mathrm{dom}(J^A_\gamma)\land x=_Xx'\rightarrow J^A_\gamma x=_XJ^A_{\gamma}x'\big)\\
\forall \gamma^1>_\mathbb{R}0,{\gamma'}^1>_\mathbb{R}0,x^X\big(\tilde{\rho}>_\mathbb{R}-\gamma\land x\in\mathrm{dom}(J^A_\gamma)\\
\qquad\qquad\qquad\land x\in\mathrm{dom}(J^A_{\gamma'})\land \gamma=_\mathbb{R}\gamma'\rightarrow J^A_\gamma x=_XJ^A_{\gamma'}x\big).
\end{cases}
\]
\item[(9)] $J^A_\gamma$ satisfies the alternative notion of being $\alpha$-conically nonexpansive for $\alpha=\frac{1}{2(\tilde{\rho}/\gamma+1)}$ if $\tilde{\rho}>-\gamma$ (on its domain), i.e.
\begin{align*}
&\forall \gamma^1,x^X,y^X\bigg(\gamma>_\mathbb{R}0\land\tilde{\rho}>_\mathbb{R}-\gamma\land x\in\mathrm{dom}(J^A_\gamma)\land y\in\mathrm{dom}(J^A_\gamma)\\
&\qquad\qquad\rightarrow \frac{2}{2(\tilde{\rho}/\gamma+1)}\langle J^A_{\gamma}x-_XJ^A_{\gamma}y,(x-_X J^A_{\gamma}x)-_X(y-_X J^A_{\gamma}y)\rangle_X\\
&\qquad\qquad\qquad\quad\geq_\mathbb{R}\left(1-\frac{2}{2(\tilde{\rho}/\gamma+1)}\right)\norm{(x-_XJ^A_{\gamma}x)-_X(y-_XJ^A_{\gamma}y)}_X^2\bigg).
\end{align*}
\item[(10)] $J^A_\gamma$ is $\alpha$-conically nonexpansive for $\alpha=\frac{1}{2(\tilde{\rho}/\gamma+1)}$ if $\tilde{\rho}>-\gamma$, i.e. $\alpha\in (0,\infty)$ and
\[
(1-\alpha^{-1})Id+\alpha^{-1}J^A_\gamma\text{ is nonexpansive (on its domain)}.
\]
\item[(11)] $J^A_\gamma$ is nonexpansive if $\tilde{\rho}\geq-\gamma/2$ (on its domain), i.e.
\begin{align*}
&\forall\gamma^1,x^X,y^X\Big(\gamma>_\mathbb{R}0\land\tilde{\rho}\geq_\mathbb{R}-\gamma/2\land x\in\mathrm{dom}(J^A_\gamma)\\
&\qquad\qquad\qquad\land y\in\mathrm{dom}(J^A_\gamma)\rightarrow\norm{J^A_\gamma x-_XJ^A_\gamma y}_X\leq\norm{x-_Xy}_X\Big).
\end{align*}
\item[(12)] $J^A_\gamma$ satisfies the alternative notion of being $\alpha$-averaged for $\alpha=\frac{1}{2(\tilde{\rho}/\gamma+1)}$ if $\tilde{\rho}\geq-\gamma/2$ (on its domain), i.e.
\begin{align*}
&\forall \gamma^1,x^X,y^X\bigg(\gamma>_\mathbb{R}0\land\tilde{\rho}\geq_\mathbb{R}-\gamma/2\land x\in\mathrm{dom}(J^A_\gamma)\land y\in\mathrm{dom}(J^A_\gamma)\\
&\qquad\qquad\rightarrow \left(1-\frac{1}{2(\tilde{\rho}/\gamma+1)}\right)\norm{(x-_XJ^A_\gamma x)-_X(y-_XJ^A_\gamma y)}^2_X\\
&\qquad\qquad\qquad\quad\leq_\mathbb{R}\frac{1}{2(\tilde{\rho}/\gamma+1)}\left(\norm{x-_Xy}^2_X-_\mathbb{R}\norm{J^A_\gamma x-_XJ^A_\gamma y}^2_X\right)\bigg).
\end{align*}
\item[(13)] $J^A_\gamma$ is $\alpha$-averaged for $\alpha=\frac{1}{2(\tilde{\rho}/\gamma+1)}$ if $\tilde{\rho}>-\gamma/2$, i.e. $\alpha\in (0,1)$ and
\[
(1-\alpha^{-1})Id+\alpha^{-1}J^A_\gamma\text{ is nonexpansive (on its domain)}.
\]
\end{enumerate}
The items (7) - (13) hold also for $\mathcal{U}^{*\omega}_p$ in a revised form: as before, any $\tilde{\rho}>_\mathbb{R}-\gamma$ or $\tilde{\rho}\geq_\mathbb{R}-\gamma/2$ needs to be replaced with $\tilde{\rho}>_\mathbb{R}-\gamma/2$.
\end{proposition}
\begin{proof}
The proofs from before, that is those of Propositions \ref{pro:respropaccr} and \ref{pro:accImplMono} as well as of Proposition \ref{pro:UFundTheorems}, carry over to the setting here. In any case where the axiom (II) was applied there, the replacement axiom (II$'$) is now applicable as we always suitably assume $x\in\mathrm{dom}(J^A_\gamma)$. Similarly for (II$_1$) and (II$'_1$).
\end{proof}
Next, we move to a characterization of extensionality for the various systems. The key new element here is that extensionality is equivalent to a combination of the strong resolvent axiom known from the total systems (now, of course, in a partial variant) \emph{together} with the version of the resolvent axiom involving the strong form of the formalization of the domain from the introduction of the partial systems.
\begin{theorem}\label{thm:extequivVT}
Over $\mathcal{V}^\omega_p$, the following are equivalent:
\begin{enumerate}
\item[(1)] Extensionality of $A$, i.e.
\[
\forall x^X,y^X,{x'}^X,{y'}^{X}\left( x=_X x'\land y=_Xy'\rightarrow \chi_Axy=_0 \chi_Ax'y'\right)
\]
\item[(2)] The strong domain axiom, i.e.
\begin{align*}
&\forall \gamma^1,x^X\big(\gamma>_\mathbb{R}0\land \exists y^X,z^X(z\in Ay\land x=_Xy+_X\gamma z)\\
&\qquad\qquad\qquad\qquad\qquad\rightarrow \gamma^{-1}(x-_XJ^A_{\gamma}x)\in A(J^A_{\gamma}x)\big),
\end{align*}
together with the strong partial resolvent axiom, i.e.
\[
\forall x^X,p^X,\gamma^1\left(\gamma>_\mathbb{R}0\land x\in\mathrm{dom}(J^A_\gamma)\land p=_XJ^A_\gamma x\rightarrow \vert\gamma\vert^{-1}(x-_Xp)\in Ap\right).
\]
\item[(3)] The strong domain axiom for $\gamma=1$, i.e.
\[
\forall x^X\left(\exists y^X,z^X(z\in Ay\land x=_Xy+_X z)\rightarrow (x-_XJ^A_{1}x)\in A(J^A_{1}x)\right),
\]
together with the strong partial resolvent axiom for $\gamma=1$, i.e.
\[
\forall x^X,p^X\left(x\in\mathrm{dom}(J^A_1)\land p=_XJ^A_{1} x\rightarrow (x-_Xp)\in Ap\right).
\]
\end{enumerate}

Further, the strong domain axiom implies strong extensionality of $J^A$, i.e.
\begin{align*}
\forall x^X,{x'}^X,\gamma^1,{\gamma'}^1\left(\gamma>_\mathbb{R}0\land x\in\mathrm{dom}(J^A_\gamma)\land\gamma=_\mathbb{R}\gamma'\land x=_Xx'\rightarrow J^A_\gamma x=_XJ^A_{\gamma'}x'\right).
\end{align*}
The same statement holds for $\mathcal{T}^\omega_p$.

Over $\mathcal{U}^\omega_p$, extensionality is similarly equivalent to the strong domain axiom for all $\gamma>0$ with $\tilde{\rho}>-\gamma$, i.e.
\begin{align*}
&\forall \gamma^1,x^X\big(\gamma>_\mathbb{R}0\land\tilde{\rho}>_\mathbb{R}-\gamma\\
&\qquad\qquad\qquad\land \exists y^X,z^X(z\in Ay\land x=_Xy+_X \gamma z)\rightarrow \gamma^{-1}(x-_XJ^A_{\gamma}x)\in A(J^A_{\gamma}x)\big),
\end{align*}
together with the strong partial resolvent axiom for all $\gamma>0$ with $\tilde{\rho}>-\gamma$, i.e.
\[
\forall x^X,p^X\left(\gamma>_\mathbb{R}0\land\tilde{\rho}>_\mathbb{R}-\gamma\land p=_XJ^A_{\gamma} x\rightarrow \gamma^{-1}(x-_Xp)\in Ap\right),
\]
and as before, the strong domain axiom implies strong extensionality of $J^A_\gamma$ if $\tilde{\rho}>-\gamma$, i.e.
\begin{align*}
&\forall x^X,{x'}^X,\gamma^1,{\gamma'}^1\big(\gamma>_\mathbb{R}0\land\tilde{\rho}>_\mathbb{R}-\gamma\land  x\in\mathrm{dom}(J^A_\gamma)\\
&\qquad\qquad\qquad\qquad\land\gamma=_\mathbb{R}\gamma'\land x=_Xx'\rightarrow J^A_\gamma x=_XJ^A_{\gamma'}x'\big).
\end{align*}
The claims hold for $\mathcal{U}^{*\omega}_p$ if all the assumptions $\tilde{\rho}>_\mathbb{R}-\gamma$ are replaced by $\tilde{\rho}>_\mathbb{R}-\gamma/2$, respectively.
\end{theorem}
\begin{proof}
We start with the equivalences:
\begin{enumerate}
\item [$(1)\Rightarrow (2)$] Let $A$ be extensional. We first show the strong domain axiom. Let $\gamma> 0$ and suppose $\exists y,z(z\in Ay\land x=y+\gamma z)$. Then $\exists y(\gamma^{-1}(x-y)\in Ay)$ by extensionality of $A$. Thus $x\in\mathrm{dom}(J^A_\gamma)$ and therefore (II$'$) yields $\vert\gamma\vert^{-1}(x-J^A_\gamma x)\in A(J^A_\gamma x)$.

Now, let $\gamma> 0$ and suppose $p=_XJ^A_\gamma x$ as well as $x\in\mathrm{dom}(J^A_\gamma)$. The latter gives $\gamma^{-1}(x-J^A_\gamma x)\in A(J^A_\gamma x)$ by axiom (II$'$). Extensionality of $A$ gives $\gamma^{-1}(x-p)\in Ap$.
\item [$(2)\Rightarrow (3)$] Clear by the rule of extensionality.
\item [$(3)\Rightarrow (1)$] Let $u\in Ax$ and suppose $u=_Xu'$ as well as $x=_Xx'$. Thus $x'+u'=_Xx+u$. By the strong domain axiom and the extensionality rule, we get $1^{-1}((x'+u')-J^A_1(x'+u'))\in A(J^A_1(x'+u'))$. Thus $x'+u'\in\mathrm{dom}(J^A_1)$ and similarly for $x+u$. Now, the rule of extensionality gives $ 1^{-1}((u+x)-x)\in Ax$ and Proposition \ref{pro:partialFundProp}, (1) gives $x=_XJ^A_1(x+u)$. Proposition \ref{pro:partialFundProp}, (4) now gives (as $x+u,x'+u'\in\mathrm{dom}(J^A_1)$) that $x'=_XJ^A_1(x'+u')$. The strong resolvent axiom implies $(x'+u'-x')\in Ax'$ and thus $u'\in Ax'$, again by the rule of extensionality.
\end{enumerate}
Now we show that the strong domain axiom implies the strong extensionality of $J^{\chi_A}$. For that, assume $\gamma> 0$ and that $x\in\mathrm{dom}(J^A_\gamma)$ as well as $\gamma=\gamma'$ and $x=_Xx'$. Then we have $\exists y(\gamma^{-1}(x-y)\in Ay)$. Therefore, we have
\[
x'=_Xx=_Xy+\gamma\gamma^{-1}(x-y)=_Xy+\gamma'\gamma^{-1}(x-y).
\]
By the strong domain axiom, we get
\[
\gamma'^{-1}(x'-J^A_{\gamma'}x')\in A(J^A_{\gamma'}x').
\]
From (II$'$), we get $\gamma^{-1}(x-J^A_\gamma x)\in A(J^A_\gamma x)$. We get by accretivity that
\[
0=\norm{(J^A_\gamma x-J^A_{\gamma'}x')+\gamma(\gamma^{-1}(x-J^A_\gamma x)-\gamma'^{-1}(x'-J^A_{\gamma'}x'))}\geq\norm{J^A_\gamma x-J^A_{\gamma'}x'}.
\]
The proofs for the claims for $\mathcal{U}^\omega_p$ are easy modifications of the above.
\end{proof}
The previous Propositions \ref{pro:resolventprop} and \ref{pro:resolventpropU} can also be modified for these partial resolvents.
\begin{proposition}\label{pro:fundrespropPartial}
$\mathcal{V}^\omega_p$ proves:
\begin{enumerate}
\item[(1)] $\begin{cases}
\forall \gamma^1,\lambda^1,x^X\big(\gamma>_\mathbb{R}0\land\lambda>_\mathbb{R}0\land x\in\mathrm{dom}(J^A_\lambda)\\
\qquad\qquad\qquad\rightarrow J^A_\lambda x=_X J^A_{\gamma}\left(\frac{\gamma}{\lambda} x+_X\left(1-\frac{\gamma}{\lambda}\right)J^A_\gamma x\right)\big).
\end{cases}$
\item[(2)] $\begin{cases}
\forall \gamma^1,\lambda^1,x^X\bigg(\gamma>_\mathbb{R}0\land\lambda>_\mathbb{R}0\land x\in\mathrm{dom}(J^A_\gamma)\land x\in\mathrm{dom}(J^A_\lambda)\\
\qquad\qquad\qquad\qquad\qquad\rightarrow\norm{x-_XJ^A_\gamma x}_X\leq_\mathbb{R}\left(2+\frac{\gamma}{\lambda}\right)\norm{x-_XJ^A_\lambda x}_X\bigg).
\end{cases}$
\end{enumerate}
A fortiori, the same claim holds for $\mathcal{T}^\omega_p$. Similarly, $\mathcal{U}^\omega_p$ proves:
\begin{enumerate}
\item[(3)] $\begin{cases}
\forall \gamma^1,\lambda^1,x^X\big(\gamma>_\mathbb{R}0\land\lambda>_\mathbb{R}0\land\tilde{\rho}>_\mathbb{R}-\gamma\land x\in\mathrm{dom}(J^A_\lambda)\\
\qquad\qquad\qquad\qquad\rightarrow J^A_\lambda x=_X J^A_{\gamma}\left(\frac{\gamma}{\lambda} x+_X\left(1-\frac{\gamma}{\lambda}\right)J^A_\gamma x\right)\big).\end{cases}$
\item[(4)] $\begin{cases}
\forall \gamma^1,\lambda^1,x^X\bigg(\gamma>_\mathbb{R}0\land\lambda>_\mathbb{R}0\land\tilde{\rho}\geq_\mathbb{R}-\frac{\gamma}{2}\land x\in\mathrm{dom}(J^A_\gamma)\\
\qquad\qquad\land x\in\mathrm{dom}(J^A_\lambda)\rightarrow\norm{x-_XJ^A_\gamma x}_X\leq_\mathbb{R}\left(2+\frac{\gamma}{\lambda}\right)\norm{x-_XJ^A_\lambda x}_X\bigg).
\end{cases}$
\end{enumerate}
As before, items (3), (4) also hold for $\mathcal{U}^{*\omega}_p$ under the appropriate modifications: (3) holds with $\tilde{\rho}>_\mathbb{R}-\gamma/2\land\tilde{\rho}>_\mathbb{R}-\lambda/2$ instead of $\tilde{\rho}>_\mathbb{R}-\gamma$ and (4) holds with $\tilde{\rho}>_\mathbb{R}-\gamma/2\land\tilde{\rho}>_\mathbb{R}-\lambda/2$ instead of $\tilde{\rho}\geq_\mathbb{R}-\gamma/2$.
\end{proposition}
The proofs are essentially the same as before and we thus omit it.

Similar remarks as before (see Remarks \ref{rem:resolventprop} and \ref{rem:resolventpropU}) regarding alternative axiomatizations of the partial systems also apply here but we omit the details.
\begin{remark}
Similar comments as made in Remarks \ref{rem:yosidaremark} and \ref{rem:yosidaremarkcomon} also apply here regarding the range of the systems and, in particular, all the previously mentioned results for the Yosida approximate extend to this partial setting in the appropriate ways.
\end{remark}
\section{Extensions motivated by mathematical practice}
As mentioned in the introduction, many concrete applications require extensions of the previously introduced systems to deal with certain conditions or moduli required in that particular situation. We discuss some possible extensions in that vein and give pointers to the case studies thus recognizable as applications of the upcoming bound extraction theorems. As mentioned before, further discussions regarding the following, and other, extensions and quantitative notions, both from the perspective of the upcoming metatheorems and their role in previous case studies involving set-valued operators, will be given in \cite{KP2022}.
\subsection{Range conditions}\label{sec:rangecond}
In various algorithmic approaches to problems involving accretive or (generalized) monotone operators, instead of requiring that an operator is maximal (which in many cases is only assumed to ensure that the resolvents are total), one often just requires that the operator fulfills a certain \emph{range condition}, i.e. a condition ensuring that the domains of the resolvents are large enough, such that some particular iteration scheme is well-defined. In that way, a range condition is really a minimal setup for many algorithmic considerations in that context.

We consider these range conditions, and the details surrounding their formalization, in the context of the well-known proximal point algorithm \cite{Mar1970,Roc1976} which is a particularly simple but instructive setup for studying them. This algorithm takes the form of the iteration
\[
x_0\in\mathrm{dom} A,\; x_{n+1}=J^A_{\gamma_n}x_n\tag{$\dagger$}
\]
for some sequence $(\gamma_n)$ with $\gamma_n>0$ for all $n$. A range condition for the operator $A$ which suffices to sustain this iteration is then, e.g., given by stipulating
\[
\mathrm{dom}A\subseteq \bigcap_{n\in\mathbb{N}}(Id+\gamma_n A)(\mathrm{dom A}).
\]
It is easy to see that by adding this assumption, the iteration given by $(\dagger)$ is well-defined.

When we naively formalize the above range condition (assuming for simplicity that the sequence $\gamma_n$ is given by a closed term $\gamma_{\cdot}$ of type $1(0)$), we end up with the sentence
\[
\forall n^0,x^X,y^X\exists z^X,w^X\left(y\in Ax\rightarrow \left(w\in Az\land x=_Xz+_X\gamma_n w\right)\right).
\]
Note that the inner matrix $y\in Ax\rightarrow \left(w\in Az\land x=_Xz+_X\gamma_n w\right)$ is a universal formula and therefore, it may be enticing to think about ways to bound the quantifiers $\exists z^X,w^X$ to turn this into an axiom of type $\Delta$ as defined in \cite{Koh2008} (see also Section \ref{sec:boundExtractionThms} for a precise definition) as these have a simple monotone functional interpretation and thus can be added to the previous systems while still allowing for bound extraction results.

Bounding $z$ leads us to consider what we shall call \emph{bounded range conditions} by which we shall mean range conditions which only require the above statement to be fulfilled on a certain closed ball in the sense of 
\[
\mathrm{dom}A\cap \overline{B}_L(a)\subseteq \bigcap_{n\in\mathbb{N}}(Id+\gamma_n A)(\mathrm{dom A}\cap \overline{B}_L(a)).
\]
In a way, this is how the range conditions are used in the context of the proximal point algorithm (and other iteration schemes) as they are only ever needed for the elements $x_n$ from the above iteration ($\dagger$) which, under the assumption of being bounded, all lie in some ball.

This kind of bounded range condition can be brought into the required form. First note that we can naively express the above range condition by restricting the quantifiers over $z$ and $x$ to be bounded by some constant $L$ via
\[
\forall n^0,x^X\preceq_X L1_X,y^X\exists z^X\preceq_X L1_X\exists w^X\left(y\in Ax\rightarrow (w\in Az\land x=_Xz+_X\gamma_n w)\right).
\]
with $\preceq_X$ defined by $x\preceq_X y:=\norm{x}_X\leq_\mathbb{R}\norm{y}_X$.

Now, we can even provide a bound for $w$ in terms of $L$: As $x=z+\gamma_n w$, we have
\[
\norm{w}=\frac{\norm{x-z}}{\vert\gamma_n\vert}\leq\frac{\norm{x}+\norm{z}}{\vert\gamma_n\vert}\leq\frac{2L}{\vert\gamma_n\vert}.
\]
As we assume $\gamma_n> 0$, the $\gamma_n$ come equipped with a modulus $\alpha_n\in\mathbb{N}$ such that $\gamma_n>2^{-\alpha_n}$ (see in particular the discussion from Section \ref{sec:sysoffinittypeintro} on how reciprocals are formally treated in $\mathrm{WE}$-$\mathrm{PA}^\omega$). Thus, using these moduli (represented by a function $\alpha$ of type $1$), we obtain
\[
\norm{w}\leq\frac{2L}{2^{-\alpha_n}}=L\cdot 2^{\alpha_n+1}.
\]
Thus, we can formulate the above range condition by bounding $w$ via
\begin{align*}
&\forall n^0,x^X\preceq_X L1_X,y^X\exists z^X\preceq_X L1_X\\
&\qquad\qquad\qquad\exists w^X\preceq_XL2^{\alpha_n+1}1_X\left(y\in Ax\rightarrow (w\in Az\land x=_Xz+_X\gamma_n w)\right).
\end{align*}
This is not yet of the form $\Delta$ as we have an additional premise in the form of $x\preceq_X L1_X$ which adds another, in particular non-bounded, existential quantifier to a prenexiation of the formula. This can be avoided by moving to an intensional version of $x\preceq_X L1_X$ via the following term construction (as used in \cite{GuK2016,KL2012} for the case of $L=1$): for $x^X$, we write 
\[
\tilde{x}^L:=\frac{Lx}{\max_\mathbb{R}\{\norm{x}_X,L\}}.
\]
Then we can consider 
\begin{align*}
&\forall n^0,x^X,y^X\exists z^X\preceq_X L1_X\exists w^X\preceq_XL2^{\alpha_n+1}1_X\\
&\qquad\qquad\qquad\left(y\in A(\tilde{x}^L)\rightarrow (w\in Az\land \tilde{x}^L=_Xz+_X\gamma_n w)\right)
\end{align*}
where we have internalized the bound on $x$ via $\tilde{x}^L$. This is now of the form $\Delta$ and can be allowed as an assumption in the upcoming bound extraction theorems. 

It is also worth noting that in the context of the $\Delta$ axioms, we of course can be much more general regarding the bounds on $z$ which can even be of the form $z\preceq_X tnxy$ for a closed term $t$ (which may be used to formulate many other types of domains besides closed balls).\\

Using this formal approach to range conditions in the context of our systems, we can recognize the previously mentioned case studies \cite{Koh2020,Koh2021b,Koh2021} on the proximal point algorithm (in the context of various types of operators $A$) as particular applications of the upcoming metatheorems. While the former just studies $\rho$-comonotone operators in Hilbert spaces, the latter two consider accretive operators in certain Banach spaces, in particular uniformly convex in the former and uniformly convex and uniformly smooth Banach spaces in the latter case. These therefore need, besides the previously mentioned formalized range conditions, additional moduli of uniform convexity and of uniform smoothness which can however also be treated in the context of bound extraction theorems as discussed already in, e.g., \cite{Koh2008}.
\subsection{Majorizable operators}\label{sec:majop}
Proofs which make essential use of representatives $y\in Ax$ for $x\in\mathrm{dom}A$ (i.e. $Ax\neq\emptyset$) can be treated by providing a suitable witnessing (Skolem) functional for the statement 
\[
\forall x^X\exists y^X(x\in\mathrm{dom}A\rightarrow y\in Ax)
\]
which can immediately be treated by adding a further constant $a$ of type $X(X)$ together with a defining axiom like
\[
\forall x^X(x\in\mathrm{dom}A\rightarrow ax\in Ax)\tag{$\mathrm{NE}$}
\]
where we write $x\in\mathrm{dom}A:=\exists y^X\left(y\in Ax\right)$. Such a witnessing functional $a$ can take many forms depending on the particular application scenario (which might require additional axioms).

In any way however, such a functional then of course requires majorizing data if used in the bound extraction theorems. Precise definitions for (strong) majorizability (where we omit the prefix `strong' in the following) and related notions will follow in the next section but we want to discuss the notion of a majorant for the functional $a$ here already: a function $f$ of type $1$ is called a \emph{majorant} for $a$ if it is non-decreasing, i.e. $n\geq m$ implies $fn\geq fm$, and it satisfies
\[
n\geq\norm{x}\rightarrow fn\geq\norm{ax}.
\]

Thus, any witnessing functional $a$ for an operator $A$ can only be treated in the context of the bound extraction theorems if there is \emph{at least one} choice which is majorizable. We thus propose the following (in the above sense minimal) definition:
\begin{definition}
An operator $A$ is called \emph{majorizable} if there exists a choice for $a$ satisfying $(\mathrm{NE})$ which is majorizable.
\end{definition}

A first thing to note is that there are non-majorizable operators $A$ and we want to give a quick example here: Consider the partial convex function $\varphi:(0,\pi/2)\to\mathbb{R}$, $x\mapsto\tan(x)$. As $\varphi$ is differentiable, we get that the subgradient of $\varphi$ is given by
\[
x\in (0,\pi/2)\mapsto\partial\varphi(x):=\left\{\frac{1}{\cos^2 x}\right\}.
\]
It is well known that $\partial\varphi(x)$ is a monotone operator and in particular can be extended to the whole domain by setting
\[
A:x\mapsto\begin{cases}\partial\varphi(x)&\text{if }x\in (0,\pi/2),\\\emptyset&\text{otherwise},\end{cases}
\]
which preserves monotonicity. It can be easily seen that $A$ is not majorizable.

The main exploited feature of the $\varphi$ defined above is that it is unbounded on a bounded subset and thus in particular not majorizable itself (if identified with some continuation from $(0,\pi/2)$ to $\mathbb{R}$). Interestingly, this majorizability of $\varphi$ is in some sense all that is needed to even guarantee a strong form of majorizability for the particular case of monotone operators of the form $\partial\varphi$. 

For a given function $\varphi:X\to\mathbb{R}$, we say that $\varphi$ is \emph{bounded on bounded sets} if $\varphi(\overline{B}_n(0))$ is bounded for any $n$. The following is then immediate:
\begin{proposition}
For a given $\varphi:X\to\mathbb{R}$, $(\vert \varphi\vert)_\circ$ is majorizable as a type $1(X)$ functional if, and only if $\varphi$ is bounded on bounded sets where $(\vert \varphi\vert)_\circ$ is defined by $(\vert \varphi\vert)_\circ(x):=(\vert \varphi(x)\vert)_\circ$ with $(\cdot)_\circ$ defined as in Section \ref{sec:sysoffinittypeintro}
\end{proposition}
This notion of $\varphi$ being bounded on bounded sets extends from $\varphi$ to $\partial\varphi$ in many situations as the following result shows. For this, we similarly introduce a notion of boundedness for set-valued operators $A$: we say that $A$ is \emph{bounded on bounded sets} if $A(\overline{B}_n(0))=\bigcup_{x\in\overline{B}_n(0)}Ax$ is bounded for any $n$.
\begin{proposition}[\cite{BC2017}, Proposition 16.20]
Let $\varphi:X\to\mathbb{R}$ be continuous and convex on a Hilbert space $X$. Then, the following are equivalent:
\begin{enumerate}
\item[(1)] $\varphi$ is bounded on bounded sets.
\item[(2)] $\varphi$ is Lipschitz continuous on every bounded set.
\item[(3)] $\mathrm{dom}\,\partial \varphi=X$ and $\partial \varphi$ is bounded on bounded sets.
\end{enumerate}
\end{proposition}
Now, an operator being bounded on bounded sets can be recognized as a majorizability assumption on $A$ in disguise.
\begin{proposition}\label{pro:bobsMajorizable}
$A$ is bounded on bounded sets if, and only if
\begin{align*}
&\exists {a^*}^{0(0)}\forall a^{X(X)}\Big(\forall x^X(x\in\mathrm{dom}A\rightarrow ax\in Ax)\\
&\qquad\qquad\qquad\land\forall x^X(x\not\in\mathrm{dom}A\rightarrow\norm{ax}=0)\rightarrow a^*\gtrsim a\Big).
\end{align*}
\end{proposition}
The proof is rather immediate and we thus omit it.

\smallskip

Therefore $A$ is bounded on bounded sets if, and only if there is a \emph{uniform} majorant $a^*$ for any $a$ satisfying $(\mathrm{NE})$ for $A$ and which is sufficiently well behaved on $(\mathrm{dom}A)^c$ (where the behavior is anyhow not important). In that case, we call $A$ \emph{uniformly majorizable}. Combined, we then in particular have for a convex and continuous $\varphi$ that majorizability of $\vert\varphi\vert$ implies uniform majorizability of $\partial\varphi$.

\smallskip

In particular, in the light of the ubiquity of the notion of `bounded on bounded sets' in the literature on monotone operators, the above proposition provides a valuable insight on the proof theoretic nature of that assumption. Further, an analysis of proofs involving such selection functionals actually in many situations only needs a majorizable and not uniformly majorizable operator and it can thus be expected that this assumption of $A$ being bounded on bounded sets can often be weakened in a proof theoretic analysis to that of simple majorizability. A concrete example for this is given in \cite{Pis2022b} in the context of an analysis for an algorithm of Moudafi \cite{Mou2015}. More connections between (uniform) majorizability and other quantitative notions for set-valued operators (especially in the context of the Br\'ezis-Haraux theorem \cite{BH1976} and its use in the analysis provided by Kohlenbach \cite{Koh2019b} for Bauschke's proof of the zero displacement conjecture \cite{Bau2003}) will be given in \cite{KP2022}.
\subsection{The minimal norm selection functional}
We want to discuss one particular selection functional which occurs often in the literature and which can actually be treated already whenever one has a majorizable operator. This operator is $A^\circ x$ which returns the element of minimal norm in $Ax$ (as long as $Ax\neq\emptyset$). This operator is particularly well-known in the context of nonsmooth analysis, e.g. being central to the dual formulation of the well-known bundle method (see, e.g., \cite{SZ1992}), as subgradients of minimal norm are used to select descent directions in nonsmooth contexts.

$A^\circ x$ is usually defined (see, e.g., \cite{BC2017}) as the projection $P_{Ax}0$ (which always exists in the context of Hilbert spaces as, e.g., follows from Proposition 20.36 and Theorem 3.16 in \cite{BC2017}). For that reason, we focus here on the special case where our space is a Hilbert space and $A$ is maximally monotone as this implies a nice characterization of the projection in terms of the inner product:
\begin{lemma}[\cite{BC2017}, Theorem 3.16]
If $X$ is a Hilbert space and $C\subseteq X$ is nonempty, closed and convex, then for any $x,p\in X$:
\[
p=P_Cx\text{ iff }\left(p\in C\text{ and }\forall q\in C\left(\langle q-p,x-p\rangle\leq 0\right)\right).
\]
\end{lemma}
Using this implicit characterization of $A^\circ x$ as $P_{Ax}0$, this operator can be simulated by adding a constant $A^\circ_X$ of type $X(X)$ together with the following axioms:
\begin{enumerate}
\item[$(\mathrm{Y1})$] $\forall x^X(x\in\mathrm{dom}A\rightarrow A^\circ_X x\in Ax)$,
\item[$(\mathrm{Y2})$] $\forall x^X,y^X(y\in Ax\rightarrow \langle y-A^\circ_X x,-A^\circ_X x\rangle\leq 0)$.
\end{enumerate}
We write $(\mathrm{Y})$ as a shorthand for the axioms $(\mathrm{Y1})$ and $(\mathrm{Y2})$ as above.

By $(\mathrm{Y1})$, $A^\circ_X$ satisfies $(\mathrm{NE})$ for $A$ which is special in the way that it is actually the minimal realizer of the respective axiom in the sense of the following proposition:
\begin{proposition}\label{pro:Acircmaj}
$A^\circ_X$ is majorizable iff $A$ is majorizable.
\end{proposition}
Thus, proofs using $A^\circ$ can indeed be treated in the context of majorizable operators. Moreover, via the axioms $(\mathrm{Y})$, we can develop a substantial part of the theory of $A^\circ$. For example, one can prove (in weak fragments of $\mathcal{A}^\omega[X,\langle\cdot,\cdot\rangle]$ already), that projections characterized by the inner product condition are unique and that $A^\circ$ is consequently the unique element of minimal norm in the sense that $\mathcal{T}^\omega+(\mathrm{Y})$ proves that
\[
\forall x^X,z^X\left(z\in Ax\land \norm{z}_X=\norm{A^\circ_X x}_X\rightarrow A^\circ_X x=_Xz\right).
\]
Using the upcoming bound extraction theorems, one can even provide suitable quantitative versions of those results (which was crucially used in a recent case study given in \cite{Pis2022b}). Note also that while completeness of the space and maximality of the operator are necessary to guarantee the existence of $A^\circ$, both properties are not provable in the system. If the proof uses these properties in an essential way, besides of ensuring existence of $A^\circ$, one has to add quantitative forms of them to the system. The next subsection provides a discussion for this regarding maximality and for possibilities regarding the completeness of the space, see Remark \ref{rem:extensionsMeta}.

The main example of application for majorizable operators was recently considered in \cite{Pis2022b} for a quantitative treatment of a convergence result due to Moudafi \cite{Mou2015}. Concretely, given two maximally monotone operators $T,S$ on a real Hilbert space $X$, \cite{Mou2015} considers the sequence
\[
x_{n+1}=J^S_{\mu_n}(x_n+\mu_n T_{\lambda_n}x_n)
\]
with initial point $x_0$ and parameters $\mu_n,\lambda_n>0$. Under suitable conditions, this sequence converges to a point $x^*\in X$ with
\[
T(x^*)\cap S(x^*)\neq\emptyset.
\]
As $T^\circ$ is strongly connected with the Yosida approximate, in particular through $T_{\lambda_n}x\to T^\circ x$ for $x\in\mathrm{dom}T$ and $\lambda_n\to 0$ (see \cite{BC2017}, Corollary 23.46), the operator $T^\circ$ features prominently in the proofs given in \cite{Mou2015} and consequently also in the analysis presented in \cite{Pis2022b} which thus requires a treatment of $T^\circ$ in the sense of the above additional constants and axioms. Also, the notion of $A$ being bounded on bounded sets features prominently in \cite{Mou2015}. As already mentioned before, the analysis weakens this assumption to simple majorizability as guided by the metatheorems which provides an example of a result where the quantitative analyses provided by this general logical approach actually yielded a qualitative improvement in the result. For further discussions of this issue, see \cite{Pis2022b}.

\subsection{Moduli of uniform continuity}\label{sec:modunifcont}
To treat problems which use some form of extensionality of $A$ in an essential way, one has to consider a corresponding (uniform) quantitative version of the respective extensionality statement. These usually come in the form of a modulus of uniform continuity (see in particular the discussions in \cite{Koh2008} for various perspectives on this issue). In the case of set-valued operators $A$, one choice might be to consider uniform continuity for $A$ w.r.t. the Hausdorff metric in the sense of \cite{MN2001} where the Hausdorff metric $H$ is defined as
\[
H(P,Q):=\max\left\{\sup_{p\in P}\inf_{q\in Q}\norm{p-q},\sup_{q\in Q}\inf_{p\in P}\norm{p-q}\right\}
\]
for closed non-empty sets $P,Q$. To motivate this (where we follow the discussion presented in \cite{KP2020}) consider the following `version' of extensionality
\[
\forall x,y\in\mathrm{dom} A(x=y\rightarrow H(Ax,Ay)=0).
\]
Guided by the monotone function interpretation, this has a uniform quantitative version asserting the existence of a \emph{modulus of uniform continuity of $A$ w.r.t. the Hausdorff metric}, i.e. the existence of an $\omega:\mathbb{N}\to\mathbb{N}$ such that
\[
\forall k\in\mathbb{N}\forall x,y\in\mathrm{dom} A\left(\norm{x-y}<\frac{1}{\omega(k)+1}\rightarrow H(Ax,Ay)\leq\frac{1}{k+1}\right).
\]
The above definition of the Hausdorff metric requires the use of infima and is therefore not immediately definable in the systems. Further, the restriction to $\mathrm{dom}A$ prohibits the comparison between elements $x$ with $Ax=\emptyset$ and thus the above is not even a real quantitative version of the extensionality statement. However, this restriction is necessary for $H(Ax,Ay)$ to be well-defined. For those reasons, Kohlenbach and Powell propose the following weakening in  \cite{KP2020}: instead of using the metric $H$, one introduces a \emph{Hausdorff-like predicate} $H^*$ defined via
\[
H^*[P,Q,\varepsilon]:=\forall p\in P\exists q\in Q\left(\norm{p-q}\leq\varepsilon\right)
\]
and stipulates uniform continuity for $A$ w.r.t. that predicate via the existence of a \emph{modulus} $\varpi:\mathbb{N}\to\mathbb{N}$ such that
\[
\forall k\in\mathbb{N}\forall x,y\in X\left(\norm{x-y}<\frac{1}{\varpi(k)+1}\rightarrow H^*\left[Ax,Ay,\frac{1}{k+1}\right]\right)\tag{$\dagger$}
\]
We can actually recognize this as the uniform quantitative version (guided by the monotone functional interpretation) of the following approximate weakening of the extensionality principle
\[
\forall x^X,y^X\left(x=_Xy\rightarrow\forall k^0 H^*\left[Ax,Ay,\frac{1}{k+1}\right] \right)
\]
which can be used in place of the usual extensionality principle in some situations, e.g., whenever the rest of the proof after the application is extensional in the variables. Moreover, $(\dagger)$ can be added to the previous systems by an axiom of the form $\Delta$. More precisely, we have $\norm{w}\leq\norm{z}+\norm{z-w}$ and therefore $\norm{w}\leq\norm{z}+\frac{1}{k+1}$ whenever $\norm{z-w}\leq\frac{1}{k+1}$ and $(\dagger)$ can then be expressed by
\begin{align*}
&\forall k^0,x^X,y^X,z^X\exists w^X\preceq_X\left(\norm{z}+\frac{1}{k+1}\right)1_X\bigg(\norm{x-y}<\frac{1}{\varpi(k)+1}\\
&\qquad\qquad\qquad\qquad\qquad\qquad\qquad\land z\in Ax\rightarrow w\in Ay\land\norm{z-w}\leq\frac{1}{k+1}\bigg)\tag{$\mathrm{UC}^*$}
\end{align*}
which is of the form $\Delta$ (where one notes that the inner matrix is purely universal with quantifiers of low enough type) and can be added to the previous systems together with an additional constant $\varpi$.

\smallskip

Examples of the use of this axiom in the context of proof mining appear in, e.g., \cite{KP2020,Pis2022b}. In \cite{KP2020}, the authors analyze (among various other results) a theorem due to \cite{MN2001} which requires a uniformly continuous operator $A$ and the analysis provided in \cite{KP2020} can be seen as an application of the upcoming metatheorems (modulo some additional considerations). In particular, uniform continuity is transformed into a modulus $\varpi$ of uniform continuity w.r.t. $H^*$ which can be treated as detailed above.

The additional considerations mentioned before relate to the treatment of the notion of uniform quasi-accretivity due to \cite{GF2005}, which is upgraded in the analysis to a modulus $\Theta$ of uniform accretivity at zero as originally introduced in \cite{KKA2015}, i.e. a function such that
\[
\forall\varepsilon,K>0\forall (x,u)\in\mathrm{gra}A\left(\norm{x-q}\in [\varepsilon,K]\rightarrow \exists j\in J(x-q)(\langle u,j\rangle\geq\Theta_K(\varepsilon)\right).
\]
Based on a suitable treatment of the normalized duality mapping $J$ (see e.g. \cite{KL2012}), it should be possible to phrase the above quantitative notion of uniform accretivity at zero as an axiom of type $\Delta$ but we do not explore this here further.

\smallskip

A further discussion on the weak extensionality principle and the corresponding notion of uniform continuity w.r.t $H^*$ regarding its use in the context of the analysis of Moudafi's algorithm \cite{Mou2015} will be given in \cite{Pis2022b}.
\section{Bound extraction theorems}\label{sec:boundExtractionThms}
We now establish the proof mining metatheorems for the theories $\mathcal{V}^\omega/\mathcal{T}^\omega/{\mathcal{U}}^{*\omega}$ and the partial variants in the vein of \cite{GeK2008,Koh2005,Koh2008}. Our proof follows the general outline from \cite{Koh2008}. As an abbreviation, we write $\mathcal{C}^\omega$ in the following for one of the partial or total systems. Further, we write $\mathcal{C}^{\omega-}$ for $\mathcal{C}^\omega$ \emph{without} $\mathrm{QF}$-$\mathrm{AC}$ and $\mathrm{DC}$. By $(\mathrm{BR})$, we denote the schema of \emph{simultaneous bar-recursion}, going back to the seminal work of Spector \cite{Spe1962}, in all types from $T^X$ (see, e.g., \cite{Koh2008}).

The basis for the upcoming metatheorems, as well as for the previously established ones in the literature, is the utilization of \emph{G\"odel's functional interpretation} (going back to G\"odel's work \cite{Goe1958}, but we mainly use the presentations from \cite{Koh2008,Tro1973}) in combination with a negative translation (which also goes back to G\"odel \cite{Goe1933} and Gentzen, although unpublished in the latter case, but we rely on a version by Kuroda \cite{Kur1951}). We recall the definitions of those interpretations here.
\begin{definition}[\cite{Goe1958,Tro1973}]
The \emph{Dialectica interpretation} $A^D=\exists\underline{x}\forall\underline{y} A_D(\underline{x},\underline{y})$ of a formula $A$ in the language of $\mathcal{C}^\omega$ is defined via the following recursion on the structure of the formula:
\begin{enumerate}
\item[(1)] $A^D:=A_D:=A$ for $A$ being a prime formula.
\end{enumerate}
If $A^D=\exists\underline{x}\forall\underline{y} A_D(\underline{x},\underline{y})$ and $B^D=\exists\underline{u}\forall\underline{v} B_D(\underline{u},\underline{v})$, we set
\begin{enumerate}
\item[(2)] $(A\land B)^D:=\exists\underline{x},\underline{u}\forall\underline{y},\underline{v}(A\land B)_D$\\ where $(A\land B)_D(\underline{x},\underline{u},\underline{y},\underline{v}):=A_D(\underline{x},\underline{y})\land B_D(\underline{u},\underline{v})$,
\item[(3)] $(A\lor B)^D:=\exists z^0\underline{x},\underline{u}\forall\underline{y},\underline{v}(A\land B)_D$\\ where $(A\lor B)_D(z^0,\underline{x},\underline{u},\underline{y},\underline{v}):=(z=0\rightarrow A_D(\underline{x},\underline{y}))\land (z\neq 0\rightarrow B_D(\underline{u},\underline{v}))$,
\item[(4)] $(A\rightarrow B)^D:=\exists\underline{U},\underline{Y}\forall\underline{x},\underline{v}(A\land B)_D$\\ where $(A\rightarrow B)_D(\underline{U},\underline{Y},\underline{x},\underline{v}):=A_D(\underline{x},\underline{Y}\underline{x}\underline{v})\land B_D(\underline{U}\underline{x},\underline{v})$,
\item[(5)] $(\exists z^\tau A(z))^D(z,\underline{x},\underline{y}):=\exists z,\underline{x}\forall\underline{y}(\exists z^\tau A(z))_D$\\ where $(\exists z^\tau A)_D:=A_D(\underline{x},\underline{y},z)$,
\item[(6)] $(\forall z^\tau A)^D:=\exists\underline{X}\forall z,\underline{y}(\forall z^\tau A(z))_D$\\ where $(\forall z^\tau A)_D(\underline{X},z,\underline{y}):=A_D(\underline{X}z,\underline{y},z)$.
\end{enumerate}
\end{definition}
\begin{definition}[\cite{Kur1951}]
The \emph{negative translation} of $A$ is defined by $A':=\neg\neg A^*$ where $A^*$ is defined by the following recursion on the structure of $A$:
\begin{enumerate}
\item[(1)] $A^*:= A$ for prime $A$;
\item[(2)] $(A\circ B)^*:= A^*\circ B^*$ for $\circ\in\{\land,\lor,\rightarrow\}$;
\item[(3)] $(\exists x^\tau A)^*:=\exists x^\tau A^*$;
\item[(4)] $(\forall x^\tau A)^*:=\forall x^\tau \neg\neg A^*$.
\end{enumerate}
\end{definition}
Following \cite{Koh2008}, we introduce a certain class of formulas already mentioned before: by $\Delta$ we in the following denote a set of formulas of the form
\[
\forall\underline{a}^{\underline{\delta}}\exists\underline{b}\preceq_{\underline{\sigma}}\underline{r}\underline{a}\forall\underline{c}^{\underline{\gamma}}F_{qf}(\underline{a},\underline{b},\underline{c})
\]
where $F_{qf}$ is quantifier-free, the types in $\underline{\delta}$, $\underline{\sigma}$ and $\underline{\gamma}$ are admissible and $\underline{r}$ are tuples of closed terms of appropriate type. Here, $\preceq$ is defined by recursion on the type via
\begin{enumerate}
\item[(1)] $x\preceq_0 y:=x\leq_0 y$,
\item[(2)] $x\preceq_X y:=\norm{x}_X\leq_\mathbb{R}\norm{y}_X$,
\item[(3)] $x\preceq_{\tau(\xi)} y:=\forall z^\xi(xz\preceq_\tau yz)$.
\end{enumerate}
Given such a set $\Delta$, we write $\widetilde{\Delta}$ for the set of all Skolem normal forms
\[
\exists\underline{B}\preceq_{\underline{\sigma}(\underline{\delta})}\underline{r}\forall\underline{a}^{\underline{\delta}}\forall\underline{c}^{\underline{\gamma}}F_{qf}(\underline{a},\underline{B}\underline{a},\underline{c})
\]
for any $\forall\underline{a}^{\underline{\delta}}\exists\underline{b}\preceq_{\underline{\sigma}}\underline{r}\underline{a}\forall\underline{c}^{\underline{\gamma}}F_{qf}(\underline{a},\underline{b},\underline{c})$ in $\Delta$.
\begin{lemma}[essentially \cite{GeK2008,Koh2005}]\label{lem:ndinterpretation}
Let $F$ be an arbitrary formula in $\mathcal{C}^\omega$ with only the variables $\underline{a}$ free. Then the rule
\[
\begin{cases}\mathcal{C}^\omega\vdash F(\underline{a})\Rightarrow\\
\mathcal{C}^{\omega-}+(\mathrm{BR})\vdash\forall\underline{a},\underline{y}(F')_D(\underline{t}\underline{a},\underline{y},\underline{a})\end{cases}
\]
holds where $\underline{t}$ is a tuple of closed terms of $\mathcal{C}^{\omega-}+(\mathrm{BR})$ which can be extracted from the respective proof.
\end{lemma}
We wrote `essentially \cite{GeK2008,Koh2005}' as the proofs presented in \cite{GeK2008,Koh2005} are only formulated for our base systems $\mathcal{A}^\omega[X,\norm{\cdot}]$ and $\mathcal{A}^\omega[X,\langle\cdot,\cdot\rangle]$ but it is immediately clear that they extend to the new constants and systems as all the axioms are purely universal.

The central concept for formulating the quantitative bounds obtained by the metatheorems is that of majorization in the sense of the extension due to \cite{GeK2008,Koh2005} of strong majorization of Bezem \cite{Bez1985} (which in turn builds on Howard's majorizability \cite{How1973}) to the new types in $T^X$. In that way, majorants of objects with types from $T^X$ will be objects with types from $T$ related by the following projection:
\begin{definition}[\cite{GeK2008}]
Define $\widehat{\tau}\in T$, given $\tau\in T^X$, by recursion on the structure via
\[
\widehat{0}:=0,\;\widehat{X}:=0,\;\widehat{\tau(\xi)}:=\widehat{\tau}(\widehat{\xi}).
\]
\end{definition}
The majorizability relation $\gtrsim_\rho$ is then defined by recursion on the type along with the corresponding structure $\mathcal{M}^{\omega,X}$ of all (strongly) majorizable functionals of finite type as defined in \cite{GeK2008,Koh2005}:
\begin{definition}[\cite{GeK2008,Koh2005}]
Let $(X,\norm{\cdot})$ be a non-empty normed space. The structure $\mathcal{M}^{\omega,X}$ and the majorizability relation $\gtrsim_\rho$ are defined by
\[
\begin{cases}
M_0:=\mathbb{N}, n\gtrsim_0 m:=n\geq m\land n,m\in\mathbb{N},\\
M_X:= X, n\gtrsim_X x:= n\geq\norm{x}\land n\in M_0,x\in M_X,\\
x^*\gtrsim_{\tau(\xi)}x:=x^*\in M_{\widehat{\tau}}^{M_{\widehat{\xi}}}\land x\in M_\tau^{M_\xi}\\
\qquad\qquad\qquad\land\forall y^*\in M_{\widehat{\xi}},y\in M_\xi(y^*\gtrsim_\xi y\rightarrow x^*y^*\gtrsim_\tau xy)\\
\qquad\qquad\qquad\land\forall y^*,y\in M_{\widehat{\xi}}(y^*\gtrsim_{\widehat{\xi}}y\rightarrow x^*y^*\gtrsim_{\widehat{\tau}}x^*y),\\
M_{\tau(\xi)}:=\left\{x\in M_\tau^{M_\xi}\mid \exists x^*\in M^{M_{\widehat{\xi}}}_{\widehat{\tau}}:x^*\gtrsim_{\tau(\xi)}x\right\}.
\end{cases}
\]
Correspondingly, the structure $\mathcal{S}^{\omega,X}$ is defined as the full set-theoretic type structure via $S_0:=\mathbb{N}$, $S_X:= X$ and
\[
S_{\tau(\xi)}:=S_{\tau}^{S_{\xi}}.
\]
\end{definition}
For an inner product space, the structures $\mathcal{S}^{\omega,X}$ and $\mathcal{M}^{\omega,X}$ are defined via the norm induced by the inner product.

Now, majorization behaves as expected for functionals with multiple arguments (represented by their `curryied' variants) as the following lemma shows:
\begin{lemma}[\cite{GeK2008,Koh2005}, see also Kohlenbach \cite{Koh2008}, Lemma 17.80]\label{lem:majlemma}
Let $\xi=\tau\xi_k\dots\xi_1$. For $x^*:M_{\widehat{\xi_1}}\to(M_{\widehat{\xi_2}}\to\dots\to M_{\widehat{\tau}})\dots)$ and $x:M_{\xi_1}\to (M_{\xi_2}\to\dots\to M_\tau)\dots)$, we have $x^*\gtrsim_\xi x$ iff
\begin{enumerate}
\item[(a)] $\forall y_1^*,y_1,\dots,y_k^*,y_k\left(\bigwedge_{i=1}^k(y^*_i\gtrsim_{\xi_i}y_i)\rightarrow x^*y^*_1\dots y^*_k\gtrsim_\tau xy_1\dots y_k\right)$ and 
\item[(b)] $\forall y_1^*,y_1,\dots,y_k^*,y_k\left(\bigwedge_{i=1}^k(y^*_i\gtrsim_{\widehat{\xi_i}}y_i)\rightarrow x^*y^*_1\dots y^*_k\gtrsim_{\widehat{\tau}} x^*y_1\dots y_k\right)$.
\end{enumerate}
\end{lemma}
The proof of the main bound extraction result now relies on a combination of functional interpretation and negative translation together with subsequent majorization. The following lemma gives the main result for the latter ingredient (akin to, e.g., Lemma 9.9 in \cite{GeK2008}).
\begin{lemma}\label{lem:majmainresult}
Let $(X,\norm{\cdot})$ be a normed space, $A$ an m-accretive operator and $J^A_\gamma$ its resolvent with parameter $\gamma>0$. Then $\mathcal{M}^{\omega,X}$ is a model of $\mathcal{V}^{\omega-}+(\mathrm{BR})$ (for a suitable interpretation of the additional constants). Moreover, for any closed term $t$ of $\mathcal{V}^{\omega-}+(\mathrm{BR})$, one can construct a closed term $t^*$ of $\mathcal{A}^\omega+(\mathrm{BR})$ such that
\begin{align*}
&\mathcal{M}^{\omega,X}\models\forall n^0,m^0,l^0,k^0\bigg(n\geq_\mathbb{R}\norm{c_X-J^A_{\tilde{\gamma}}c_X}_X\land m\geq_0m_{\tilde{\gamma}}\\
&\qquad\qquad\qquad\qquad\land\; l\geq_\mathbb{R}\vert\tilde{\gamma}\vert\land k\geq_\mathbb{R}\norm{c_X}_X\rightarrow t^*(n,m,l,k)\gtrsim t\bigg).
\end{align*}
Further, the same claim holds for $\mathcal{V}^\omega$ replaced with 
\begin{enumerate}
\item[(1)] $\mathcal{T}^\omega$ where the conclusion is then drawn over inner product spaces with maximally monotone $A$ or $\mathcal{T}^\omega$ extended by $A^\circ_X$ and the axioms $(\mathrm{Y})$ if the space is further a Hilbert spaces and $A$ a majorizable operator where $t^*$ then depends on an additional parameter $g^1$ with the assumption $g\gtrsim_{X(X)}A^\circ_X$ added to the premise,
\item[(2)] $\mathcal{U}^{*\omega}$ where the conclusion is drawn over inner product spaces with maximally $\rho$-comonotone $A$ where $\rho>-r_{\tilde{\gamma}}/2$ and $t^*$ depends on two additional parameters $o^0,p^0$ with the assumptions $o\geq_0 n_{\tilde{\gamma}}$ and $p\geq_\mathbb{R}\vert\tilde{\rho}\vert$ added to the premise,
\item[(3)] the partial systems $\mathcal{V}^\omega_p$, $\mathcal{T}^\omega_p$ and $\mathcal{U}^{*\omega}_p$ where the conclusion is drawn over the appropriate spaces and operators, assuming that $\bigcap_{\gamma>0}\mathrm{dom}(J^A_\gamma)\neq\emptyset$ (where the intersection is constructed over all $\gamma$ where additionally $\rho>-\gamma/2$ in the case of $\mathcal{U}^{*\omega}_p$ and $t^*$ depends on two additional parameters $o^0,p^0$ as above),
\item[(4)] any of the above system extended by $\varpi$ and the axiom $(\mathrm{UC}^*)$ if the operator is uniformly continuous w.r.t. $H^*$ where the term $t^*$ then depends on an additional parameter $h^1$ with additional assumption $h\gtrsim_{0(0)}\varpi$.
\end{enumerate} 
\end{lemma}
\begin{proof}
We only verify the result for the systems $\mathcal{V}^\omega$ and $\mathcal{U}^{*\omega}$ and only for the new constants $\chi_A$, $J^{\chi_A}$, $\tilde{\gamma}$, $c_X$, $m_{\tilde{\gamma}}$ as well as potentially $n_{\tilde{\gamma}}$ and $\tilde{\rho}$. The rest follows as in \cite{Koh2008}, Lemma 17.85. We first deal with the non-partial case. The designated interpretation of the constant $\chi_A$ in the model $\mathcal{M}^{\omega,X}$ is given by
\[
[\chi_A]_\mathcal{M}:=\lambda x,y\in X.\begin{cases}0^0&\text{if }y\in Ax,\\1^0&\text{if }y\not\in Ax,\end{cases}
\]
while the constant for the resolvent in interpreted by
\[
[J^{\chi_A}]_\mathcal{M}:=\lambda\alpha\in\mathbb{N}^\mathbb{N},x\in X.\begin{cases}J^A_{r_\alpha}x&\text{if }r_\alpha> 0,\\0&\text{otherwise},\end{cases}
\]
where $r_\alpha$ is the real represented by $\alpha$ as before. We set $[\tilde{\gamma}]_\mathcal{M}:=(\lambda)_\circ$ and $[m_{\tilde{\gamma}}]_\mathcal{M}:=m_\lambda$ for some real $\lambda$ and natural $m_\lambda$ with $\lambda\geq 2^{-m_\lambda}$. Lastly, in the case of the total systems, we define $[c_X]_\mathcal{M}:=c$ for some arbitrary $c\in X$.

These constants are then majorizable (and their interpretations thus belong to $\mathcal{M}^{\omega,X}$). For $\chi_A$,
\[
\lambda x^0,y^0.1\gtrsim \chi_A
\]
is immediate by the previous Lemma \ref{lem:majlemma} and by the fact that $\mathcal{V}^\omega$/$\mathcal{T}^\omega$/${\mathcal{U}}^{*\omega}$ prove that $\chi_Axy\leq_01$. For $J^{\chi_A}$, given $n\geq\norm{c_X-J^A_{\tilde{\gamma}}c_X}$ and $k\geq\norm{c_X}$ as well as $m\geq m_{\tilde{\gamma}}$, we obtain
\[
\lambda\alpha^1,{x^*}^0.x^*+2k+(2+2^{m}(\alpha(0)+1))n\gtrsim J^{\chi_A}
\]
To see that, let $\gamma^1,x^X$ be given and let $x^*\gtrsim x$, i.e. $x^*\geq\norm{x}$, as well as $\alpha\gtrsim\gamma$. Then in particular (with similar reasoning as in \cite{Koh2008}, Lemma 17.85)
\[
\alpha(0)+1\geq\gamma(0)+1\geq\vert\gamma\vert.
\]
Now, we have (even provably in $\mathcal{V}^\omega$/$\mathcal{T}^\omega$/${\mathcal{U}}^{*\omega}$) that for $r_\gamma> 0$:
\begin{align*}
\norm{J^A_\gamma x}&\leq\norm{x-c_X}+\norm{J^A_\gamma c_X}&&\text{(nonexpansivity)}\\
&\leq\norm{x}+\norm{c_X}+\norm{c_X-J^A_{\gamma}c_X}+\norm{c_X}&&\\
&\leq\norm{x}+2\norm{c_X}+\left(2+\frac{\gamma}{\tilde{\gamma}}\right)\norm{c_X-J^A_{\tilde{\gamma}}c_X}&&\text{(Proposition \ref{pro:resolventprop}/Remark \ref{rem:resolventprop})}\\
&\leq x^*+2k+(2+2^{m}(\alpha(0)+1))n
\end{align*}
For $r_\gamma\leq 0$, we get that $J^A_\gamma x=0$. Thus, $\norm{J^A_\gamma x}=0\leq x^*+2k+(2+2^{m}(\alpha(0)+1))n$ in that case as well. This implies majorizability using Lemma \ref{lem:majlemma}. Lastly, $\lambda i^0.(l+1)\gtrsim_{0(0)}\tilde{\gamma}$, $m\geq_0m_{\tilde{\gamma}}$ and $k\gtrsim_{X} c_X$ are immediate by the assumptions on $l,m$ and $k$, respectively.

In the case of $\mathcal{U}^{*\omega}$, we set $[\tilde{\rho}]_\mathcal{M}=(\rho)_\circ$. Further, in that case the interpretation of the resolvent changes to
\[
[J^{\chi_A}]_\mathcal{M}:=\lambda\alpha\in\mathbb{N}^\mathbb{N},x\in X.\begin{cases}J^A_{r_\alpha}x&\text{if }r_\alpha> 0\text{ and }\rho>- r_\alpha/2,\\0&\text{otherwise},\end{cases}
\]
and $\tilde{\gamma}$, $m_{\tilde{\gamma}}$ as well as $n_{\tilde{\gamma}}$ are now interpreted by $\lambda$, $m_\lambda$ and $n_\lambda$ such that $\lambda\geq 2^{-m_\lambda}$ as well as $\rho\geq -\lambda/2+2^{-n_\lambda}$. Then, the above argument for majorization still goes through for $r_\gamma>0$ and $\rho>-r_\gamma/2$, however noting Proposition \ref{pro:resolventpropU} and Remark \ref{rem:UstarProp}. Majorants for $m_{\tilde{\gamma}}$, $n_{\tilde{\gamma}}$, $\tilde{\gamma}$, $c_X$ and $\tilde{\rho}$ are immediate as before.

\smallskip

In the partial case, let $c\in\mathrm{dom}(J^A_\gamma)$ for any $\gamma> 0$ (with $\rho>-\gamma/2$ in the case of ${\mathcal{U}}^{*\omega}_p$) and define $[c_X]_\mathcal{M}:=c$. Now, the resolvent is interpreted by
\[
[J^{\chi_A}]_\mathcal{M}:=\lambda\alpha\in\mathbb{N}^\mathbb{N},x\in X.\begin{cases}J^A_{r_\alpha}x&\text{if }r_\alpha> 0\text{ and }x\in\mathrm{dom}(J^A_{r_\alpha}),\\0&\text{otherwise},\end{cases}
\]
in the case of $\mathcal{V}^\omega_p$ and by
\[
[J^{\chi_A}]_\mathcal{M}:=\lambda\alpha\in\mathbb{N}^\mathbb{N},x\in X.\begin{cases}J^A_{r_\alpha}x&\text{if }r_\alpha> 0,\rho>-r_\alpha/2\text{ and }x\in\mathrm{dom}(J^A_{r_\alpha}),\\0&\text{otherwise},\end{cases}
\]
in the case of $\mathcal{U}^{*\omega}_p$. The argument for majorizability of $J^A_\gamma$ is the same as before, just restricting to $x\in\mathrm{dom}(J^A_{ r_\gamma})$ and using nonexpansivity on the domain and Proposition \ref{pro:fundrespropPartial}. The other constants are interpreted and majorized as before.

Note that the corresponding extensions of $\mathcal{M}^{\omega,X}$ to the new constants are indeed models of the theory as none of the axioms for $J^{\chi_A}$ prescribe behavior of the resolvent for $\gamma\leq 0$ (or $\gamma$ with $\rho\leq -\gamma/2$ in the case of ${\mathcal{U}}^{*\omega}$).

\smallskip

Lastly, for the potential constants $A^\circ_X$ or $\varpi$: for $A^\circ$, naturally $A^\circ$ exists on $\mathrm{dom}A$ for a maximally monotone $A$ on a Hilbert space (which is assumed) and we set $[A^\circ]_\mathcal{M}(x):=A^\circ x$ for $x\in\mathrm{dom}A$ and $[A^\circ]_\mathcal{M}(x):=0$ otherwise which is majorizable as $A$ is assumed to be majorizable in that case (see Proposition \ref{pro:Acircmaj}).

For an operator $A$ which is uniformly continuous w.r.t. $H^*$, $\varpi$ is naturally interpreted by a respective modulus which is majorizable as it is a functional of type $0(0)$.
\end{proof}
We now formulate the bound extraction theorem. The potential additional axioms $\Delta$ are treated in spirit of the so-called monotone functional interpretation due to \cite{Koh1996} (and conceptually already to \cite{Koh1990,Koh1992a,Koh1992b}). 

We say in spirit of the monotone functional interpretation as we actually don't use a monotone variant of the functional interpretation but treat the functional interpretation part and the subsequent majorization separately, which nevertheless allows one to treat the axioms of type $\Delta$ similarly as in
\begin{enumerate}
\item[(1)] Corollary 6.5 (see also Theorem 3.30 in \cite{Koh2005}) is derived from Theorem 6.3 in \cite{GeK2008} (or Corollary 17.70 from Theorem 17.69 in \cite{Koh2008}),
\item[(2)] Corollary 5.14 is obtained from Theorem 5.13 in \cite{GuK2016}. 
\end{enumerate}
For that, we need the following lemma:
\begin{lemma}[\cite{GuK2016}, Lemma 5.11]
Let $\Delta$ be a set of formulas of the form considered before. Then $\mathcal{S}^{\omega,X}\models\Delta$ implies $\mathcal{M}^{\omega,X}\models\widetilde{\Delta}$.
\end{lemma}
\begin{proof}
The proof given in \cite{GuK2016} for Lemma 5.11 carries over.
\end{proof}
\begin{theorem}\label{thm:metatheoremMonOp}
Let $\tau$ be admissible, $\delta$ be of degree $1$ and $s$ be a closed term of $\mathcal{V}^\omega$ of type $\sigma(\delta)$ for admissible $\sigma$. Let $B_\forall(x,y,z,u)$/$C_\exists(x,y,z,v)$ be $\forall$-/$\exists$-formulas of $\mathcal{V}^{\omega}$ with only $x,y,z,u$/$x,y,z,v$ free. If
\[
\mathcal{V}^{\omega}+\Delta\vdash\forall x^\delta\forall y\preceq_\sigma s(x)\forall z^\tau\left(\forall u^0 B_\forall(x,y,z,u)\rightarrow\exists v^0 C_\exists(x,y,z,v)\right),
\]
then one can extract a partial functional $\Phi:S_{\delta}\times S_{\widehat{\tau}}\times\mathbb{N}^4\rightharpoonup\mathbb{N}$ which is total and (bar-recursively) computable on $\mathcal{M}_\delta\times\mathcal{M}_{\widehat{\tau}}\times\mathbb{N}^4$ and such that for all $x\in S_\delta$, $z\in S_\tau$, $z^*\in S_{\widehat{\tau}}$ and all $n,m,l,k\in\mathbb{N}$, if $z^*\gtrsim z$, $n\geq_\mathbb{R}\norm{c_X-J^A_{\tilde{\gamma}}(c_X)}_X$, $m\geq_0m_{\tilde{\gamma}}$, $l\geq_\mathbb{R}\vert\tilde{\gamma}\vert$ and $k\geq_\mathbb{R}\norm{c_X}_X$, then
\begin{align*}
&\mathcal{S}^{\omega,X}\models\forall y\preceq_\sigma s(x)\big(\forall u\leq\Phi(x,z^*,n,m,l,k) B_\forall(x,y,z,u)\\
&\qquad\qquad\qquad\qquad\qquad\rightarrow\exists v\leq\Phi(x,z^*,n,m,l,k)C_\exists(x,y,z,v)\big)
\end{align*}
holds whenever $\mathcal{S}^{\omega,X}\models\Delta$ for $\mathcal{S}^{\omega,X}$ defined via any normed space $(X,\norm{\cdot})$ with $\chi_A$ interpreted by the characteristic function of an m-accretive $A$ and $J^{\chi_A}$ by corresponding resolvents $J^A_\gamma$ for $\gamma>0$ (and the other constants accordingly).

In particular:
\begin{enumerate}
\item[(1)] If $\widehat{\tau}$ is of degree $1$, then $\Phi$ is a total computable functional. 
\item[(2)] We may have tuples instead of single variables $x,u,v$ and a finite conjunction instead of a single premise $\forall u^0 B_\forall(x,u)$.
\item[(3)] If the claim is proved without $\mathrm{DC}$, then $\tau$ may be arbitrary and $\Phi$ will be a total functional on $S_\delta\times S_{\widehat{\tau}}\times\mathbb{N}^4$ which is primitive recursive in the sense of G\"odel.
\item[(4)] The claim of the above theorem as well as the items (1) - (3) from above holds similarly for
\begin{enumerate}
\item $\mathcal{T}^\omega$ where the conclusion is then drawn over inner product spaces with maximally monotone operators or $\mathcal{T}^\omega$ extended by $A^\circ_X$ and the axioms $(\mathrm{Y})$ with the additional assumption $g\gtrsim_{X(X)}A^\circ_X$, $\Phi$ depending additionally on $g$ and the conclusion being drawn over Hilbert spaces,
\item $\mathcal{U}^{*\omega}$ where the conclusion is drawn over inner product spaces with maximally $\rho$-comonotone operators and resolvents $J^A_{\gamma}$ for $\rho>-\gamma/2$ (where $\Phi$ depends on two additional parameters $o^0,p^0$ with the additional assumptions $o\geq_0 n_{\tilde{\gamma}}$ and $p\geq_\mathbb{R}\vert\tilde{\rho}\vert$),
\item $\mathcal{V}^\omega_p$, $\mathcal{T}^\omega_p$ and $\mathcal{U}^{*\omega}_p$ where the conclusion is drawn over the appropriate spaces and operators such that, in particular, $\bigcap_{\gamma>0}\mathrm{dom}J^A_\gamma\neq\emptyset$ (with $\gamma$ additionally satisfying $\rho>-\gamma/2$ and $\Phi$ depending on two additional parameters $o^0,p^0$ as above in the case of $\mathcal{U}^{*\omega}_p$),
\item the above systems extended by the constant $\varpi$ and axiom $(\mathrm{UC}^*)$ whenever $A$ is uniformly continuous w.r.t. $H^*$ where we have the additional assumption $h\gtrsim_{0(0)}\varpi$ and $\Phi$ depends additionally on $h$.
\end{enumerate}
\end{enumerate}
\end{theorem}
\begin{proof}
We only treat the case of $\mathcal{V}^\omega$. Proofs for item (4), (a) - (d) follow the same reasoning. The set $\Delta$ can be treated as in the proof of Theorem 5.13 in \cite{GuK2016}: Add the Skolem functionals $\underline{B}$ from $\widetilde{\Delta}$ to the language. Then, $\widetilde{\Delta}$ can be seen as another set of universal axioms and all the new constants are majorizable by assumption since $\underline{B}\preceq_{\underline{\sigma}(\underline{\delta})}\underline{r}$ and since $\underline{r}$ is a tuple of closed terms. Then, the following proof goes through for this extended system instead of $\mathcal{V}^\omega$ (where one has to note that Lemma \ref{lem:ndinterpretation} immediately holds for this purely universal extension as well):

\smallskip

First, assume that
\[
\mathcal{V}^{\omega}\vdash\forall z^\tau\left(\forall u^0 B_\forall(z,u)\rightarrow\exists v^0 C_\exists(z,v)\right).
\]
By assumption, $B_\forall(z,u)=\forall\underline{a} B_{qf}(z,u,\underline{a})$ and $C_\exists(z,v)=\exists\underline{b} C_{qf}(z,v,\underline{b})$ for quantifier-free $B_{qf}$ and $C_{qf}$. Thus, prenexing the above theorem of $\mathcal{V}^\omega$, we get
\[
\mathcal{V}^\omega\vdash\forall z^\tau\exists u,v,\underline{a},\underline{b}(B_{qf}(z,u,\underline{a})\rightarrow C_{qf}(z,v,\underline{b})).
\]
Using Lemma \ref{lem:ndinterpretation}, disregarding the realizers for $\underline{a},\underline{b}$ and reintroducing the quantifiers, we get closed terms $t_u,t_v$ of $\mathcal{V}^{\omega-}+(\mathrm{BR})$ such that
\[
\mathcal{V}^{\omega-}+(\mathrm{BR})\vdash\forall z^\tau(B_\forall (z,t_u(z))\rightarrow C_\exists(z,t_v(z))).
\]
By Lemma \ref{lem:majmainresult} there are closed terms $t^*_u,t^*_v$ of $\mathcal{A}^\omega+(BR)$ such that for all $n\geq\norm{c_X-J^A_{\tilde{\gamma}}(c_X)}$, $m\geq m_{\tilde{\gamma}}$, $l\geq\vert\tilde{\gamma}\vert$ and $k\geq\norm{c_X}$, we get
\[
\mathcal{M}^{\omega,X}\models t^*_u(n,m,l,k)\gtrsim t_u\land t^*_v(n,m,l,k)\gtrsim t_v\land\forall z^\tau(B_\forall(z,t_u(z))\rightarrow C_\exists(z,t_v(z)))
\]
for all normed spaces $(X,\norm{\cdot})$ and all m-accretive operators $A$ with resolvents $J^A_\gamma$ defining $\mathcal{M}^{\omega,X}$ as in Lemma \ref{lem:majmainresult}. Define 
\[
\Phi(z^*,n,m,l,k):=\max(t^*_u(n,m,l,k)(z^*),t^*_v(n,m,l,k)(z^*)).
\]
Then
\[
\mathcal{M}^{\omega,X}\models\forall u\leq\Phi(z^*,n,m,l,k) B_\forall(z,u)\rightarrow\exists v\leq\Phi(z^*,n,m,l,k) C_\exists(z,v)
\]
holds for all $n\geq\norm{c_X-J^A_{\tilde{\gamma}}(c_X)}$, $m\geq m_{\tilde{\gamma}}$, $l\geq\vert\tilde{\gamma}\vert$, $k\geq\norm{c_X}$ as well as all $z\in M_\tau$ and $z^*\in M_{\widehat{\tau}}$ with $z^*\gtrsim z$. The conclusion that $\mathcal{S}^{\omega,X}$ satisfies the same sentence can made as in the proof of Theorem 17.52 in \cite{Koh2008}.

\smallskip

For the additional $\forall x^\delta\forall y\preceq_\sigma s(x)$, let $\delta=1$ for simplicity. For $x$ of type $\tau$, we can define $x^M(y^0)=\max_\mathbb{N}\{x(i)\mid 1\leq i\leq y\}$. We get $x^M\gtrsim x$ and if $s(x)\geq_\sigma y$, then $s^*(n,m,l,k)(x^M)\gtrsim y$ where $s^*$ is as in Lemma \ref{lem:majmainresult}. Note now that the above result immediately extends to tuples $z$. Then by the above result for tuples instead of a single $z$, there now is a functional $\Phi'(x^*,y^*,z^*,n,m,l,k)$ such that
\begin{align*}
&\mathcal{S}^{\omega,X}\models\forall u\leq\Phi'(x^*,y^*,z^*,n,m,l,k)B_\forall (x,y,z,u)\\
&\qquad\qquad\qquad\rightarrow\exists v\leq\Phi'(x^*,y^*,z^*,n,m,l,k)C_\exists(x,y,z,v)
\end{align*}
for all $x\in S_{\delta},y\in S_{\sigma},z\in S_{\tau}$ where $y\preceq s(x)$ and with $x^*\gtrsim x$, $y^*\gtrsim y$, $z^*\gtrsim z$ and $n,m,l,k$ as before.
In particular, we have 
\begin{align*}
&\mathcal{S}^{\omega,X}\models\forall u\leq\Phi'(x^M,y^*,z^*,n,m,l,k)B_\forall (x,y,z,u)\\
&\qquad\qquad\qquad\rightarrow\exists v\leq\Phi'(x^M,y^*,z^*,n,m,l,k)C_\exists(x,y,z,v)
\end{align*}
for any such $x,y,z$ and $y^*,z^*$ and thus, as $y\leq_\sigma s(x)$ yields $s^*(n,m,l,k)(x^M)\gtrsim y$, we get
\begin{align*}
&\mathcal{S}^{\omega,X}\models\forall u\leq\Phi'(x^M,s^*(n,m,l,k)(x^M),z^*,n,m,l,k)B_\forall (x,y,u)\\
&\qquad\qquad\qquad\rightarrow\exists v\leq\Phi'(x^M,s^*(n,m,l,k)(x^M),z^*,n,m,l,k)C_\exists(x,y,v).
\end{align*}
Then define $\Phi(x,z^*,n,m,l,k)=\Phi'(x^M,s^*(n,m,l,k)(x^M),z^*,n,m,l,k)$.

\smallskip

Item (1) can be shown as in the proof of Theorem 17.52 from \cite{Koh2008} (see page 428). Further, (2) is immediate and (3) follows from the fact that without $\mathrm{DC}$, bar recursion becomes superfluous.
\end{proof}
\begin{remark}\label{rem:extensionsMeta}
The above results can be immediately extended to augmentations of the systems considered here by, e.g., the following:
\begin{enumerate}
\item[(1)] Further abstract metric and normed spaces which are treated simultaneously (see the discussion in \cite{Koh2008} in Section 17.6).
\item[(2)] Further constants for monotone operators and their resolvents which in particular may mix partial and non-partial resolvents.
\item[(3)] A constant $C^{X(X0)}$ which associates with every Cauchy sequence a limit to treat complete spaces and proofs which make essential use of this completeness assumption (see \cite{Koh2008}, pages 432-434, and also \cite{Sip2019} for a recent practical use in the context of $L_p$-spaces). As commented on before, this may in particular occur in proofs relying on the operator $A^\circ$ as completeness of the space is needed to guarantee its existence and thus may be occurring at other places in such proofs in an essential way.
\item[(4)] Additional $\Delta$-formulas as axioms or additional constants (where corresponding defining axioms guarantee majorizability) which are of admissible types (see the discussion in \cite{Koh2008}, Section 17.5). This in particular includes, e.g., constants for moduli of uniform convexity or uniform smoothness of the space (see \cite{Koh2008}, Section 17.3 for more details on the former).
\end{enumerate}
\end{remark}

\section*{Acknowledgments}
This paper is a revised version of parts of my master thesis \cite{Pis2022} written under the supervision of Prof. Dr. Ulrich Kohlenbach at TU Darmstadt. In that vein, I want to thank Prof. Kohlenbach. His lectures on applied proof theory have made the greatest impact on my mathematical interest and since my first meetings with him, I have immensely enjoyed (and benefited from) our various discussions and his usual precise comments, in particular also while working on my thesis and this work.

Also I want to thank Sam Sanders for providing valuable comments regarding the presentation of the subject matter.

\bibliographystyle{plain}
\bibliography{ref}

\end{document}